\documentclass[pdflatex,sn-mathphys-num]{sn-jnl}

\usepackage{graphicx}%
\usepackage{multirow}%
\usepackage{amsmath,amssymb,amsfonts}%
\usepackage{amsthm}%
\usepackage{mathrsfs}%
\usepackage[title]{appendix}%
\usepackage{xcolor}%
\usepackage{textcomp}%
\usepackage{manyfoot}%
\usepackage{booktabs}%
\usepackage[ruled,vlined,resetcount,algosection]{algorithm2e}
\usepackage{algpseudocode}%
\usepackage{listings}%
\usepackage[caption=false]{subfig}

\theoremstyle{thmstyleone}%
\newtheorem{theorem}{Theorem}
\newtheorem{proposition}[theorem]{Proposition}%

\theoremstyle{thmstyletwo}%
\newtheorem{remark}{Remark}%
\newtheorem{lemma}{Lemma}

\theoremstyle{thmstylethree}%

\DeclareMathOperator{\prox}{prox}
\DeclareMathOperator*\Id{Id}%

\begin{document}

\title[Inertial parameterized Douglas-Rachford for generalized DC programming problems]{Inertial parameterized Douglas-Rachford for generalized DC programming problems}


\author[1]{\fnm{Avinash} \sur{Dixit}}\email{avinash@kmc.du.ac.in}
\equalcont{These authors contributed equally to this work.}

\author*[2]{\fnm{Pankaj} \sur{Gautam}}\email{pankaj.gautam@amsc.iitr.ac.in}

\author[3,4]{\fnm{Sorin-Mihai} \sur{Grad}}\email{sorin-mihai.grad@uni-due.de}
\equalcont{These authors contributed equally to this work.}

\affil*[1]{\orgdiv{Department of Mathematics}, \orgname{Kirori Mal College, University of Delhi}, \orgaddress{ \city{Delhi}, \postcode{110007}, \country{India}}}

\affil[2]{\orgdiv{Department of Applied Mathematics \& Scientific Computing}, \orgname{Indian Institute of Technology Roorkee}, \orgaddress{\city{Roorkee}, \postcode{247667}, \country{India}}}

\affil[3]{\orgdiv{Faculty of Mathematics}, \orgname{University of Duisburg-Essen}, \orgaddress{\street{Thea-Leymann-Str. 9}, \city{Essen}, \postcode{45127}, \country{Germany}}}

\affil[4]{\orgdiv{Department of Mathematics}, \orgname{Babe\textcommabelow{s}-Bolyai University Cluj-Napoca}, \orgaddress{\street{Kog\u{a}lniceanu Street, No. 1}, \city{Cluj-Napoca}, \postcode{400084}, \country{Romania}}}


\abstract{We propose two new Douglas-Rachford type algorithms for solving generalized DC programming problems. More precisely, we deal with the minimization of the sum of two possibly nonsmooth convex functions from which a Fr\'echet differentiable convex function with a Lipschitz continuous gradient is subtracted, in a Hilbertian framework. We introduce an inertial Douglas-Rachford DC algorithm and a parametrized inertial Douglas-Rachford DC algorithm for solving optimization problems of this type, and study their convergence behavior. Numerical experiments illustrate the behavior of the proposed algorithms when solving sparse recovery and robust classification problems, highlighting their superior performance in comparison with their state-of-the-art counterparts.}




\maketitle

\section{Introduction}\label{section1}
Non-convex optimization tools and methods play a crucial role in solving problems arising in the real world. Many applications from different branches of science can be framed as non-convex optimization problems; see, for instance, \cite{YJS, BMN}. The lack of convexity in a function increases the computational complexity when determining its global minimizers (or suitable approximations), as one may land on local minimizers of the function. The difficulty is more serious when we are dealing with high-dimensional datasets. In the last four decades, different methods for solving nonconvex optimization problems, such as cutting plane, branch and bound, and branch and cut, which can guarantee the globality of the solutions were proposed, but with high computational costs. On the other hand, the proximal point algorithm and its variants for minimizing structured functions known as splitting methods deliver accurate results for reasonable costs.

Considering that $\phi$ and $h: \mathcal{H}\to (-\infty,+ \infty]$ are proper, convex functions, and $\mathcal{H}$ is a Hilbert space, a difference of convex (DC, from now on) optimization problem is defined as
$$    \min_{x\in \mathcal{H}} \{\phi(x)-h(x)\}. $$
Motivated by recent developments in both theory and applications (see  \cite{tao1986algorithms, ding2022inexact} for applications in dimensionality reduction, \cite{tu2020hybrid} for one in image processing, \cite{BMN} for one in logistics, and \cite{wu2021nonconvex} for one in multiple-input multiple-output), we consider in this work a class of generalized DC problems, namely
\begin{align}\label{1.1}
    \min_{x\in \mathcal{H}} \{\phi(x)+g(x)-h(x)\},
\end{align}
where $\phi, g:\mathcal{H}\to (-\infty, +\infty]$ are proper, convex and lower semicontinuous functions, and $h:\mathcal{H} \to (-\infty, +\infty)$ is a Fr\'echet differentiable convex function with an $L$-Lipschitz continuous gradient $\nabla h$, where $L>0$. Using the optimality conditions from convex optimization, problem \eqref{1.1} is equivalent, under some standard hypotheses, to the inclusion problem
\begin{align}\label{en1.3}
    \text{find} ~x\in \mathcal{H} \text{ : } 0 \in \partial \phi (x) + \partial g(x) -\nabla h(x),
\end{align}
where $\partial$ denotes the (convex) subdifferential of the corresponding function. We denote the solution set of \eqref{1.1} by $\Omega_{GDCP}$, and we call its elements \textit{stationary points} of \eqref{1.1}.

In 2022, Chuang et al. \cite{chuang2022unified} proposed a unified Douglas-Rachford (DR) algorithm that converges to a critical point to \eqref{1.1}. In \cite{YJS}, two fast versions of the DR splitting method were proposed for solving problem \eqref{1.1}, while in \cite{HDO} an algorithm derived from a three-operator splitting scheme was employed to the same end. A similar DC optimization problem was solved in \cite{BAB, PDA} by means of splitting proximal gradient algorithms; however, with the function $g$ taken Fr\'echet differentiable with a Lipschitz continuous gradient, and $h$ not necessarily smooth.

Recall that DR splitting was originally introduced by Douglas and Rachford \cite{douglas1956numerical} to find numerical solutions to heat differential equations. The convex version of \eqref{1.1}, i.e., when $h\equiv 0$, namely
\begin{align}\label{COP}
    \min_{x\in \mathcal{H}} \{\phi(x)+g(x)\}
\end{align}
was studied by Lions and Mercier \cite{lions1979splitting}, who approached it
by solving the optimality condition $0\in \partial \phi(x)+ \partial g(x)$. Their proposed DR splitting method can be written as
\begin{align*}
    x_{n+1}= \frac{1}{2}x_n+ \frac{1}{2}(2\prox_{\gamma\phi} -\Id)\circ (2\prox_{\gamma g}-\Id)(x_n),
\end{align*}
where $\gamma>0$, $\Id:\mathcal{H} \to \mathcal{H}$ is the \textit{identity mapping} and the \textit{proximal mapping} of $\phi$ of coefficient $\gamma$ is defined as
$$\prox_{\gamma \phi}:\mathcal{H} \to \mathcal{H},\
x\mapsto \arg\min_{y\in \mathcal{H}}\{\phi(y)+\frac{1}{2\gamma}\|x-y\|^2\}.$$
For solving problem \eqref{COP}, one can also use the Peaceman-Rachford (PR) splitting method, initially introduced in
\cite{peaceman1955numerical} for differential equations and extended by Lions and Mercier \cite{lions1979splitting} to  nonlinear equations, defined in our setting as follows \begin{align*}
    x_{n+1}=  (2\prox_{\gamma\phi} -\Id)\circ (2\prox_{\gamma g}-\Id) (x_n),
\end{align*}
where $\gamma>0$.
In 2019, Wang and Wang \cite{wang2019parameterized} introduced a parameterized DR (PDR) algorithm in the convex setting, namely
\begin{align*}
    x_{n+1}=\left(\left(1-\frac{1}{\alpha}\right)\Id+\frac{1}{\alpha}(\alpha~\prox_{\gamma \phi}-\Id)\circ (\alpha~ \prox_{\gamma g}-\Id)\right) (x_n),
\end{align*}
where $\alpha\in (1,2]$ is a parameter that provides great flexibility and improves the speed of the algorithm, a very relevant feature in practical applications. The limit of the sequence delivered by PDR is the unique least norm optimal solution to \eqref{COP}. In 2021, Biana and Zhang \cite{bian2021parameterized} extended the PDR splitting method to nonconvex optimization problems of type
\begin{align*}
    \min_{x \in \mathcal{H}} \{f (x)+h(x)\},
\end{align*}
where at least one of $f,h:\mathcal{H}\to (-\infty, +\infty]$ is a nonconvex function. For an initial point $z_0\in \mathcal{H}$, this splitting method is defined as
\begin{align*}
    \begin{cases}
        x_{n+1}\in \arg\min_{x}\{ f(x)+\frac{1}{2\gamma}\|x-z_{n}\|^2\}\\
        y_{n+1}\in \arg \min_{y}\{h(y)+\frac{1}{2\gamma }\|y-\alpha x_{n+1}+z_n\|^2\}\\
        z_{n+1}= z_n+ (y_{n+1}-x_{n+1}),
    \end{cases}
\end{align*}
where $\gamma>0$ is the step size and $\alpha \in ({3}/{2}, 2]$. The authors established the convergence of the PDR algorithm in the mentioned case, showing by numerical examples that it outperforms the DR and PR methods in several aspects.

In 1964, Polyak \cite{polyak1964some} used an inertial term to define the heavy ball method to minimize a convex function. The advantage of the inertial term is that it uses the two previous iterates of the algorithm to define the next one, a procedure which can increase the convergence speed of the algorithm for a suitable choice of the inertial parameters. Recently, inertial forward-backward \cite{boct2016inertial} and inertial Douglas-Rachford \cite{feng2021inertial,dixit2022convergence} methods have been analyzed for nonconvex optimization problems.

In this paper, we propose an inertial Douglas-Rachford algorithm and a parametrized inertial Douglas-Rachford algorithm to solve  generalized DC programming problems of type \eqref{1.1}. We study the convergence of the proposed algorithms and perform numerical experiments that illustrate their behavior when solving sparse recovery and robust classification problems, highlighting that they outperform their state-of-the-art counterparts from the literature. Different to the similar Douglas-Rachford type methods available in the literature for solving \eqref{1.1} (see \cite[Algorithm 2]{HDO} or \cite[Algorithm 1]{YJS}), where the inertial parameters are allowed to take only certain values, in the algorithms we propose they can be freely chosen in a given interval, providing our approach more flexibility, which, in particular, can lead to superior performance.

The organization of the paper is as follows: Section \ref{section1} introduces the DC optimization problem and briefly recalls the state-of-the-art in the related literature. Section \ref{section2} provides the necessary preliminaries for making the paper self-contained. In Section \ref{section3}, we propose an inertial Douglas-Rachford algorithm and in Section \ref{section4} a parametrized inertial Douglas-Rachford algorithm for solving generalized DC programming problems, and study their convergence behavior. In Section \ref{section5}, we highlight the outcomes of numerical experiments meant to solve sparse recovery and robust classification problems using the proposed algorithms, and compare the performance of the latter against already known iterative methods.

\section{Preliminaries}\label{section2}
In this section, we introduce the definitions and results used in the rest of the paper. We denote by $\mathbb{R}$ the space of reals, by $\mathbb{N}$ the space of natural numbers, by $\mathcal{H}$ a real Hilbert space, by $Gr(T)$ the graph of operator $T: \mathcal{H} \to 2^\mathcal{H}$, by $x_n \to x$ the strong convergence of a sequence $\{x_n\}\subseteq \mathcal{H}$ towards $x\in \mathcal{H}$, and by $x_n \rightharpoonup x$ the weak convergence of $\{x_n\}$ towards $x$. A set-valued operator $T: \mathcal{H} \to 2^\mathcal{H}$ is said to be \textit{monotone} if
\begin{align*}
 \langle x-y,u-v\rangle \geq 0 \ \ \forall (x,u), (y,v) \in Gr(T)
\end{align*}
 and \textit{maximally monotone} if there exists no monotone operator $S:\mathcal{H} \to 2^\mathcal{H}$ such that $Gr(S)$ contains $Gr(T)$. The \textit{resolvent} of a monotone operator $T:\mathcal{H}\to 2^\mathcal{H}$ is $J_T= (\Id+T)^{-1}:\mathcal{H}\to 2^\mathcal{H}$. Consider $f:\mathcal{H} \to (-\infty,+\infty]$ be a proper function, then the \textit{subdifferential} of $f$ is denoted by $\partial f: \mathcal{H} \to 2^\mathcal{H}$ and defined for $x\in \mathcal{H}$ by
\begin{align*}
    \partial f(x)= \{u: f(y) \geq f(x) + \langle y-x ,u \rangle  \  \forall \  y \in \mathcal{H} \},
\end{align*}
when $f(x)\in \mathbb{R}$, while $\partial f(x)=\emptyset$ otherwise. In particular, $J_{t\partial f} = \prox_{tf}$ for all $t>0$. Moreover, one has, for $x, u\in \mathcal{H}$, $u=\prox_f(x)$ if and only if $x-u\in \partial f(u)$.

A proper function $f:\mathcal{H}\to (-\infty, +\infty]$ is said to be \textit{$\rho$-strongly convex} if $f-\frac{\rho}{2}\|\cdot\|^2$ is convex for $\rho>0$.

The following results are necessary in our study.

\begin{lemma}  [Alvarez] \label{lem2.1} \cite[Lemma 2.3]{Alvarez2003}
    Let $\{a_n\}_n$, $\{b_n\}_n$ and $\{\theta_n\}_n$ be sequences in $[0,+\infty)$ such that for all $n\in \mathbb{N}$, we have
    $$
        a_{n+1}\le a_n+\theta_n (a_n-a_{n-1})+b_n, ~\sum_{n=1}^{+\infty} b_n< +\infty
    $$
    and there exists $\theta \in \mathbb{R}$ with $0\le \theta_n \le \theta<1$ for all $n\in \mathbb{N}$.
     Then there exists $\bar a\in [0,+\infty)$ such that $a_n\to \bar a$.
\end{lemma}
\begin{lemma} [Opial] \cite{Opial1967} \label{lem2.2}
    Consider a nonempty subset $C$ of $\mathcal{H}$. Let $\{\tau_n\}_n$ be a sequence in $\mathcal{H}$ such that the following two conditions hold
    \begin{itemize}
        \item [(i)] for all $\zeta\in C$, $\lim_{n\to +\infty}\|\tau_n-\zeta\|$ exists,
        \item [(ii)] every sequential weak cluster point of $\{\tau_n\}_n$ is in $C$.
    \end{itemize}
    Then the sequence $\{\tau_n\}_n$ converges weakly to a point in $C$.
\end{lemma}
\vspace{1ex}
Also, for $x, y, z \in \mathcal{H}$ and $\beta \in \mathbb{R}$, the following identities are known (see, e.g., \cite{BAC})
\begin{align}
 & ~~~~~ \langle x - z, y - x \rangle= \frac{1}{2} \lVert z - y \rVert^{2} -
 \frac{1}{2} \lVert x - z \rVert^{2} - \frac{1}{2} \lVert y - x \rVert^{2},
 \label{id1} \\
 & \lVert \beta x + (1-\beta) y \rVert^{2} = \beta \lVert x \rVert^{2} + (1 -
 \beta) \lVert y\rVert^{2} - \beta(1 - \beta) \lVert x - y \rVert^{2}.
 \label{id2}
\end{align}

For the remainder of the paper, we work with the generalized DC problem \eqref{1.1}, where $\phi, g:\mathcal{H}\to (-\infty, +\infty]$ are proper, convex and lower semicontinuous functions, and $h:\mathcal{H} \to (-\infty, +\infty)$ is a Fr\'echet differentiable convex function with an $L$-Lipschitz continuous gradient $\nabla h$, with $L>0$. We assume that $\Omega_{GDCP}\neq \emptyset$, i.e. the set of stationary points of \eqref{1.1} contains at least an element.

Regarding $\Omega_{GDCP}$, one has the following characterization (see \cite[Proposition 3.1]{chuang2022unified}), for any fixed $\beta >0$,
\begin{equation}\label{eq5}
w\in \Omega_{GDCP} \iff \ \exists x\in \mathcal{H}:
\left\{
\begin{array}{l}
w = \prox_{\beta \phi} (x)\\
w= \prox_{\beta g} (2w - x + \beta \nabla h(w)).
\end{array}
\right.
\end{equation}
This reveals that determining a stationary point $w$ of \eqref{1.1} amounts to finding a point $x\in \mathcal{H}$ for which \eqref{eq5} holds. For $\beta >0$, the set (see \cite{YJS, HDO})
$$\Gamma_\beta =\left\{x\in \mathcal{H}|\ \exists w\in \Omega_{GDCP}: w = \prox_{\beta \phi} (x) = \prox_{\beta g} (2w - x + \beta \nabla h(w))\right\}$$
and, for $w\in \Omega_{GDCP}$, the set
$$\Gamma_\beta (w) =\left\{x\in \mathcal{H}|\ w = \prox_{\beta \phi} (x) = \prox_{\beta g} (2w - x + \beta \nabla h(w))\right\}$$
are nonempty due to the nonemptiness of $\Omega_{GDCP}$ and \eqref{eq5}. Obviously, $\Gamma_\beta = \cup_{w\in \Omega_{GDCP}}\Gamma_\beta (w)$.

\begin{remark}\label{re1}
Even though the functions $\phi$, $g$ and $h$ are assumed only convex, one can consider them, without loss of generality, in the following investigations to be strongly convex, more precisely, $\phi$ and $g$ can be taken $\rho$-strongly convex, and $h$ $2\rho$-strongly convex, for a fixed $\rho>0$. For computational reasons, we will take $\rho>L/2$. Then
$$\phi(x)+g(x)-h(x) = \left(\phi(x) + \frac{\rho}{2} \|x\|^2\right)+\left(g(x) + \frac{\rho}{2} \|x\|^2\right)-\left(h(x) + \rho \|x\|^2\right)\ \forall x\in \mathcal{H}.$$
Note also that, when $\phi$ and $g$ are $\rho$-strongly convex (without the trick mentioned above) and $\rho>L/2$, the set $\Omega_{GDPC}$ contains exactly one element, see \cite[Proposition 3.1]{HDO}. However, this conclusion does not necessarily hold when strong convexity is added to $\phi$ and $g$ are described above.
\end{remark}

\section{Inertial Douglas-Rachford DC algorithm}\label{section3}
In this section, we propose an inertial Douglas-Rachford algorithm for solving the generalized DC optimization problem \eqref{1.1}. Afterwards, we study the convergence properties of the proposed methods.



	\begin{algorithm}[H]
		\SetAlgoLined
				
		\KwIn{\begin{enumerate}
                \item Parameter $\beta>0$;
				\item initial points $x_0 \text{ and } x_1 \in \mathcal{H}$;
				\item sequences $\{\theta_n\}_n \subseteq (0,1) \text{ and } \{\kappa_n\}_n\subseteq (0,2)$.
				
			\end{enumerate}	
		}
		\SetKwBlock{For}{For}{}
		\SetKwProg{}{}{}{}
		\For($n\in \mathbb{N}$; ){
			$u_n=x_n+\theta_n(x_n-x_{n-1})$\\
     $y_n= \prox_{\beta \phi} (u_n)$\\
     $z_n=\prox_{\beta g}(2 y_n- u_n+\beta\nabla h(y_n))$\\
     $x_{n+1}= u_n+\kappa_n(z_n-y_n)$.
			}
		\KwOut{$x_{n+1},z_n \text{ and } y_n$.}		
		\caption{Inertial Douglas-Rachford DC algorithm (IDRDC) }\label{P4Al4.11}
	\end{algorithm}


\begin{remark}\label{re2}
In the literature, one can find some iterative methods (i.e. inertial Douglas-Rachford type algorithms) that look similar to Algorithm \ref{P4Al4.11}, however, there are some subtle differences between them that are highlighted below.
\cite[Algorithm 2]{HDO} has a similar definition; however, it allows only certain values for $\{\theta_n\}_n$ that are derived iteratively, as they depend on the other generated sequences, while Algorithm \ref{P4Al4.11} allows much more flexibility, as $\{\theta_n\}_n \subseteq (0,1)$ can  even be a priori fixed. This permits additional investigations on our algorithm in order to identify optimal choices of the inertial parameters. On the other hand, the inertial DR type algorithms proposed in \cite{YJS} contain an additional sequence to the ones required by our Algorithm \ref{P4Al4.11}, whose handling increases the running time and, consequently, the associated costs, as revealed by the numerical experiments presented in Section \ref{section5}. The mentioned feature brings Algorithm \ref{P4Al4.11} closer to the original inertial DR algorithm than its counterparts. Moreover, the construction we propose makes the coefficient of $x_n$ in the inertial term (where $u_n$ is defined) belong to the interval $(1,2)$, while in \cite[Algorithm 2]{YJS} it lies in $[0, 1)$. One could also mention \cite[Algorithm 1]{chuang2022unified} in this context; however, it is a ``pure'' DR type method, without inertial steps.
\end{remark}

Before giving the convergence statement for Algorithm \ref{P4Al4.11}, we need to show a technical statement.

\begin{lemma}\label{lem3.1}
    Considering $w \in \Omega_{GDPC}$, there exists a $y\in \Gamma_\beta$ such that the sequences generated by Algorithm \ref{P4Al4.11} satisfy
     \begin{align}\label{e13}
        (1+\beta\rho)\|y_n-w\|^2\le \langle y_n-w, u_n-y \rangle\ \forall n \in \mathbb{N}.
    \end{align}
\end{lemma}
\begin{proof}
    From the way Algorithm \ref{P4Al4.11} is defined, we have for a fixed $n\geq 1$ that (via \eqref{id2})
    \begin{align}\label{eq3.10}
        &y_n=(\Id+\beta \partial \phi)^{-1}(u_n)\nonumber\\
        \Rightarrow&\frac{1}{\beta}(u_n-y_n)\in \partial \phi(y_n).
    \end{align}
    and
    $$
        z_n=(\Id+\beta \partial g)^{-1}(2y_n-u_n+\beta \nabla h(y_n))$$
       \begin{align}\label{eq6} \Rightarrow \nabla h(y_n)+\frac{1}{\beta}(2y_n-u_n-z_n)\in \partial g (z_n).
    \end{align}
    Consider $w \in \Omega_{GDPC}$. Using relation \eqref{en1.3}, we have a $y \in \mathcal{H}$ such that
    \begin{align}\label{eq3.13}
        \frac{1}{\beta}(y-w)\in \partial \phi(w),
    \end{align} and
     \begin{equation}\label{eq2}
            \frac{1}{\beta}(w-y)+\nabla h(w) \in \partial g(w).
    \end{equation}
From \eqref{eq3.13} and \eqref{eq2}, one can derive
$y\in \Gamma_\beta (w)$.
    As $\phi$ is strongly convex (see Remark \ref{re1}), 
    \begin{align*}
        \rho\|y_n-w\|^2&\le \langle y_n-w,\theta-\alpha
        \rangle\ \forall \theta\in \partial \phi (y_n)\ \forall \alpha\in \partial\phi (w).
    \end{align*}
Using relations \eqref{eq3.10} and \eqref{eq3.13}, we have
    \begin{align*}
        \rho\|y_n-w\|^2&\le \frac{1}{\beta}\langle y_n-w, u_n-y_n-y+w\rangle\\
        &=\frac{1}{\beta}\langle y_n-w, u_n-y\rangle-\frac{1}{\beta}\|y_n-w\|^2,
    \end{align*}
    which implies \eqref{e13}.
\end{proof}
The convergence statement for Algorithm \ref{P4Al4.11} follows. Note that its outcomes are consistent with the ones obtained in the literature for similar iterative methods, like the ones mentioned in Remark \ref{re2}.

\begin{theorem}
Let $\{x_n\}_n$, $\{u_n\}_n$, $\{y_n\}_n$ and $\{z_n\}_n$ be sequences generated by Algorithm \ref{P4Al4.11}.
 When $\sum_{n=1}^{+\infty}\kappa_n(2-\kappa_n)< +\infty$, if $\{x_n\}_n$ verifies $\sum_{n=1}^{+\infty}\|x_n-x_{n-1}\|^2 < +\infty$, then it and $\{u_n\}_n$ converge weakly to the same point in $\Gamma_\beta$. Moreover, the sequences $\{y_n\}_n$ and $\{z_n\}_n$ converge strongly to the same stationary point of \eqref{1.1}.               \end{theorem}
\begin{proof}
Using the definition of Algorithm \ref{P4Al4.11}, for a $w\in \Omega_{GDCP}$ and a $y\in \Gamma_\beta (w)$, one has, for a fixed $n\ge 1$,

    \begin{align}\label{e66}
        \|x_{n+1}-y\|^2&=\|u_n+\kappa_n(z_n-y_n)-y\|^2\notag\\
        &=\left\|(1-\frac{\kappa_n}{2})(u_n-y)+\frac{\kappa_n}{2}(2z_n-2y_n+u_n-y)\right\|^2\notag \\
        &=\left(1-\frac{\kappa_n}{2}\right)\|u_n-y\|^2+\frac{\kappa_n}{2}\|2z_n-2y_n+u_n-y\|^2\notag\\
        &~~-\frac{\kappa_n}{2}\left(1-\frac{\kappa_n}{2}\right)\|2z_n-2y_n\|^2\notag \\
        & = \left(1-\frac{\kappa_n}{2}\right)\|u_n-y\|^2+\frac{\kappa_n}{2}\|2z_n-2y_n+u_n-y\|^2-\kappa_n\left(2-\kappa_n\right)\|z_n-y_n\|^2.
    \end{align}
    Since $g$ is $\rho$-strongly convex, from \eqref{eq3.10} and \eqref{eq2}  we have (via Lemma \ref{lem3.1})
    \[
   \rho\|z_n-w\|^2 \le \left\langle z_n-w,\sigma_1 - \sigma_2  \right\rangle ~\forall \sigma_1 \in \partial g(z_n),  \sigma_2  \in \partial g(w)
    \]
    followed by
  \begin{align}\label{eq3.7}
        \rho\|z_n-w\|^2 &\le \left\langle z_n-w, \nabla h(y_n)-\nabla h(w)+\frac{1}{\beta}(2y_n-u_n-z_n)-\frac{1}{\beta}(w-y)\right\rangle,
    \end{align}
    which implies that (via \eqref{id1})
    \begin{align*}
        2\beta\rho\|z_n-w\|^2&\le 2\beta\langle z_n-w, \nabla h(y_n)-\nabla h(w)\rangle+2\langle z_n-w, 2y_n-z_n-u_n +y-w\rangle \\
        &\le 2\beta L\|z_n-w\|\|y_n-w\|+ 2\langle z_n-w, 2y_n-z_n-u_n +y-w\rangle \\
        & \le \beta L \|z_n-w\|^2+\beta L\|y_n-w\|^2\\
        &~~+\|z_n-w\|^2+ \|2y_n-z_n-u_n+y-w\|^2-\|2z_n-2y_n+u_n-y\|^2,
    \end{align*}
    i.e.,
      $$  \|2z_n-2y_n+u_n-y\|^2\le -2\beta\rho\|z_n-w\|^2+\beta L\|z_n-w\|^2$$ $$+\|z_n-w\|^2+\beta L\|y_n-w\|^2
        +\|2y_n-z_n-u_n+y-w\|^2
 $$  
        \begin{equation}\label{e65}
        =(1+\beta L -2\beta \rho)\|z_n-w\|^2+\beta L\|y_n-w\|^2+\|2y_n-z_n-u_n+y-w\|^2.
        \end{equation}

    Now, using \eqref{e13},
    {\allowdisplaybreaks\begin{align}\label{e67}
        \|2y_n-z_n-u_n+y-w\|^2&= \|(y_n-z_n)+(y_n-w)-(u_n-y)\|^2\notag\\
        &=\|y_n-z_n\|^2+\|y_n-w\|^2+\|u_n-y\|^2+2\langle y_n-w, y_n-z_n\rangle\notag\\
        &~~-2\langle y_n-w, u_n-y\rangle-2\langle y_n-z_n, u_n-y\rangle\notag\\
        &= \|y_n-z_n\|^2+\|y_n-w\|^2+\|u_n-y\|^2+\|z_n-y_n\|^2+\|y_n-w\|^2\notag\\
        &~~-\|z_n-w\|^2-2\langle y_n-w, u_n-y\rangle -2\langle y_n-z_n, u_n-y\rangle\notag\\
        &\le 2\|y_n-z_n\|^2+\|u_n-y\|^2-\|z_n-w\|^2-2\beta \rho\|y_n-w\|^2\notag\\
        &~~-2\langle y_n-z_n, u_n-y\rangle\notag\\
        &\leq 2\|y_n-z_n\|^2+\|u_n-y\|^2-\|z_n-w\|^2\notag\\
        &~~-(2+4\beta\rho)\|y_n-w\|^2-2\langle y_n-z_n, u_n-y\rangle.
        \end{align}}
Since $g$ is $\rho$-strongly convex, from \eqref{eq3.7}, we deduce
\begin{align*}
    2\langle z_n-w, u_n-y\rangle & \le 2\beta \langle z_n-w, \nabla h(y_n)-\nabla h(w)\rangle+2\langle z_n-w, 2y_n-z_n-w\rangle-2\beta\rho \|z_n-w\|^2,
\end{align*}
    which implies that,
    \begin{align}\label{e68}
        2\langle z_n-w, u_n-y\rangle & \le 2\beta \langle z_n-w, \nabla h(y_n)-\nabla h(w)\rangle+2\langle z_n-w, y_n-z_n\rangle\notag\\
        &~~+2\langle z_n-w, y_n-w\rangle-2\beta\rho \|z_n-w\|^2\notag\\
        &\le \beta L \|z_n-w\|^2+\beta L \|y_n-w\|^2+2\langle z_n-w, y_n-z_n\rangle\notag\\
        &~~+2\langle z_n-w, y_n-w\rangle-2\beta\rho \|z_n-w\|^2\notag\\
        &=\beta (L-2 \rho) \|z_n-w\|^2+\beta L\|y_n-w\|^2+\|y_n-w\|^2-\|z_n-w\|^2\notag\\
        &~~-\|y_n-z_n\|^2+\|y_n-w\|^2+\|z_n-w\|^2-\|z_n-y_n\|^2\notag\\
        &= \beta( L-2\rho)\|z_n-w\|^2+(2+\beta L)\|y_n-w\|^2-2\|y_n-z_n\|^2.
    \end{align}
    So, we deduce from \eqref{e66}, \eqref{e67}, \eqref{e68} and \eqref{e65} that
\begin{align} &\|x_{n+1}-y\|^2 \le \left(1-\frac{\kappa_n}{2}\right)\|u_n-y\|^2-\kappa_n(2-\kappa_n)\|z_n-y_n\|^2\nonumber\\       
           &~~~+\frac{\kappa_n}{2}\left\{ (1+\beta L-2\beta \rho)\|z_n-w\|^2+\beta L\|y_n-w\|^2+2\|y_n-z_n\|^2+\|u_n-y\|^2\right.\nonumber\\
        &~~~\left. -\|z_n-w\|^2-2\beta\rho \|y_n-w\|^2-(2+2\beta \rho)\|y_n-w\|^2+(\beta L-2\beta \rho)\|z_n-w\|^2 \right.\nonumber\\
        &~~~\left. +(2+\beta L)\|y_n-w\|^2-2\|z_n-y_n\|^2 \right\}\nonumber\\
        &~~= \left(1-\frac{\kappa_n}{2}\right)\|u_n-y\|^2-\kappa_n(2-\kappa_n)\|z_n-y_n\|^2\nonumber\\
        &~~~+\frac{\kappa_n}{2}\left\{ 2( \beta L-2\beta \rho)\|z_n-w\|^2+2(\beta L-2\beta \rho)\|y_n-w\|^2+\|u_n-y\|^2\right\}\nonumber\\
        &~~= \|u_n-y\|^2-\kappa_n\beta(2\rho-L)\{\|z_n-w\|^2+\|y_n-w\|^2\}-\kappa_n(2-\kappa_n)\|z_n-y_n\|^2\label{15}\\
        &~~=(1+\theta_n)\|x_n-y\|^2-\theta_n\|x_{n-1}-y\|^2+\theta_n(1+\theta_n)\|x_n-x_{n-1}\|^2\nonumber\\
        &~~~-\kappa_n\beta(2\rho-L)\{\|z_n-w\|^2+\|y_n-w\|^2\}-\kappa_n(2-\kappa_n)\|z_n-y_n\|^2\label{eq4}\\
        &~~=\|x_n-y\|^2+\theta_n\{\|x_n-y\|^2-\|x_{n-1}-y\|^2\}+\theta_n(1+\theta_n)\|x_n-x_{n-1}\|^2\nonumber\\
        &~~~-\kappa_n\beta (2\rho-L)\{\|z_n-w\|^2+\|y_n-w\|^2\}-\kappa_n(2-\kappa_n)\|z_n-y_n\|^2\nonumber\\
        &~~\le\|x_n-y\|^2+\theta_n\{\|x_n-y\|^2-\|x_{n-1}-y\|^2\}+\theta_n(1+\theta_n)\|x_n-x_{n-1}\|^2,\label{eq3}
      \end{align}
\noindent as $\kappa_n \in (0,2)$ for all $n\in \mathbb{N}$ and $2\rho >L$. Using Lemma \ref{lem2.1}, the sequence $\{\|x_{n+1}-y\|\}_n$ converges. Similarly, one obtains that $\{\|u_{n}-y\|\}_n$ converges, too.

  Due to the boundedness of the sequence $\{x_{n}\}_n$, 
    there exists a subsequence $\{x_{n_j}\}_j$ of $\{x_n\}_n$ 
    such that $x_{n_j}\rightharpoonup \bar x$, 
    for some $\bar x\in\mathcal{H}$. By taking into account the way $\{u_n\}_n$ is defined, one can derive $u_{n_j}\rightharpoonup \bar x$, too.

Further, taking into account that $\|x_n-x_{n-1}\|\to 0$ (otherwise, the hypothesis $\sum_{n=1}^{+\infty}\|x_n-x_{n-1}\|^2 < +\infty$ would be violated), $L<2\rho$ and $\kappa_n\in (0, 2)$, \eqref{eq4} yields the convergence to zero of $\{\| z_n-y_n\|\}_n$, $\{\| z_n-w\|\}_n$ and $\{\| y_n-w\|\}_n$. Hence, $\{y_n\}_n$ and $\{z_n\}_n$ converge strongly toward the same element $w$ of $\Omega_{GDCP}$.

 Also, we have (via \eqref{eq3.10} and \eqref{eq6}) for all $j\in \mathbb{N}$
  $$  \frac{1}{\beta}(u_{n_j}-y_{n_j})\in \partial \phi(y_{n_j})$$
    and
   $$\nabla h(y_{n_j})+\frac{1}{\beta}(2y_{n_j}-u_{n_j}-z_{n_j})\in \partial g (z_{n_j}).$$
Employing \cite[Proposition 16.36]{BAC}, one obtains
     \begin{align*}
        &\frac{1}{\beta}(\bar{x}-w)\in \partial \phi (w)\Rightarrow
        \bar{x}\in w+\beta \partial \phi (w)\\
        & \Rightarrow \bar{x}\in (\Id+\beta \partial \phi )(w) \Rightarrow w= \prox_{\beta \phi}(\bar{x}),
    \end{align*}
    as well as
    \begin{align*}
      & \nabla h(w)+\frac{1}{\beta}(w-\bar{x})\in \partial g(w)
        \Rightarrow\bar{w}-\bar{x}+\beta \partial h(\bar{w})\in \beta \partial g(\bar{w})\\
        & \Rightarrow 2 w-\bar{x}+\beta \nabla h( w)\in (\Id+\beta \partial g)(w)
    \Rightarrow w= \prox_{\beta g}(2w-\bar{x}+\beta \nabla h(w)),
   \end{align*}
so $\bar x \in \Gamma_\beta (\bar{w})$. Lemma \ref{lem2.2} yields then that $x_{n}\rightharpoonup \bar x\in \Gamma_\beta (w) \subseteq \Gamma_\beta$, and the same conclusion can be drawn for $\{u_n\}_n$.

\end{proof}

We end the section by providing a result similar to \cite[Proposition 2.5]{Alvarez2003} for the sequence $\{x_n\}_n$ generated by Algorithm \ref{P4Al4.11}, meant to provide a convergence statement for it whose hypotheses do not depend on the generated sequences but only on the employed parameters.

\begin{proposition}\label{pr1}
Let $\{x_n\}_n$, $\{u_n\}_n$, $\{y_n\}_n$ and $\{z_n\}_n$ be sequences generated by Algorithm \ref{P4Al4.11}, and assume that
\begin{equation}\label{e15}
0<\underline\kappa\le\kappa_n\le\overline\kappa<2,
\end{equation}
and
\begin{equation}\label{e16}
 \bar\theta = \sup_{n\ge 1} \theta_n <1,\ \Theta=\sum_{n=1}^{+\infty}\theta_n<+\infty.
\end{equation}
Then $\sum_{n=1}^{+\infty}\|x_n-x_{n-1}\|^2 < +\infty$.
\end{proposition}

\begin{proof}
Define
$$ a_\kappa= \underline\kappa(2-\overline\kappa)>0 $$
and
$$ b_\rho= \beta\underline\kappa(2\rho-L)>0. $$

Then from \eqref{15}
\begin{equation}\label{e20}
\|x_{n+1}- y\|^2 \leq \|u_n-y\|^2-b_\rho(\|z_n-w\|^2+\|y_n-w\|^2)-a_\kappa\|z_n-y_n\|^2.
\end{equation}
For $n\geq 1$ denote $d_n=x_n-x_{n-1}$ and $e_n=x_n-y$. Then $u_n=x_n+\theta_n d_n$ and $ u_n-y=e_n+\theta_n d_n$, so
$$ \|u_n-y\|^2 = \|e_n\|^2 +2\theta_n\langle e_n,d_n\rangle +\theta_n^2\|d_n\|^2.$$
Using
$$ 2\langle e_n,d_n\rangle \leq \|e_n\|^2+\|d_n\|^2, $$
we then get
\begin{equation}\label{e21}
\|u_n-y\|^2 \leq (1+\theta_n) \|e_n\|^2 + \theta_n(\theta_n+1)\|d_n\|^2,
\end{equation}
so \eqref{e20} yields
\begin{equation}\label{e22}
\|e_{n+1}\|^2 \leq (1+\theta_n)\|e_n\|^2 + \theta_n(\theta_n+1)\|d_n\|^2-b_\rho(\|z_n-w\|^2+\|y_n-w\|^2)-a_\kappa\|z_n-y_n\|^2.
\end{equation}
As $ d_n=e_n-e_{n-1}$, we have
$$ \|d_n\|^2 \leq 2\|e_n\|^2+2\|e_{n-1}\|^2,$$
so
$$ \|e_{n+1}\|^2 \leq (1+3\theta_n+2\theta_n^2)\|e_n\|^2 + (2\theta_n+2\theta_n^2)\|e_{n-1}\|^2.$$

Denote $ M_n:=\max\{\|e_n\|^2,\|e_{n-1}\|^2\}$. As $ \theta_n\leq\bar\theta <1$,
we obtain
$$ M_{n+1} \leq \left(1+5\theta_n+4\theta_n^2\right)M_n\leq (1+9\theta_n)M_n,$$
where we took into account $\|e_n\|^2\leq M_{n+1} \leq (1+9\theta_n)M_n$. Then $ M_{n+1} \leq M_1 \prod_{k=1}^n (1+9\theta_k)$.
Using the inequality $1+t\leq e^t$ for all $t\geq 0$, one gets
$$ \prod_{k=1}^{n}(1+9\theta_k) \leq \prod_{k=1}^{n}e^{9\theta_k} = e^{9\sum_{k=1}^{n}\theta_k},$$
so
$$ M_{n+1} \leq M_1 e^{9\sum_{k=1}^{n}\theta_k}.$$
Due to \eqref{e16}, one gets
$$M_{n+1}\leq M_1e^{9\Theta},$$
so
$$ \sup\limits_{n\ge 1} M_n<+\infty.$$
Noting that
$$ M_n=\max\{\|x_n-\bar y\|^2,\|x_{n-1}- y\|^2\}, $$
one obtains then
$$ \sup\limits_{n\ge 1}\|x_n- y\|<+\infty,$$
so, consequently,
$$ \|d_n\| = \|x_n-x_{n-1}\| \leq \|x_n- y\|+\|x_{n-1}- y\|, $$
so
$$ \sup\limits_{n\ge 1} \|d_n\| < +\infty,$$
thus $\{d_n\}_n$ is bounded. Then there exists a $C>0$ such that
$$ (\theta_n+\theta_n^2)\|d_n\|^2 \leq C\theta_n. $$
Therefore,
$$b_\rho \left( \|y_n-\bar x\|^2+ \|z_n-\bar x\|^2 \right) + a_\kappa\|z_n-y_n\|^2 \leq \|e_n\|^2-\|e_{n+1}\|^2+C\theta_n,$$

and summing from $n=1$ to an $N>0$ one gets
$$ b_\rho\sum_{n=1}^N \left( \|y_n-\bar x\|^2+ \|z_n-\bar x\|^2 \right) + a_\kappa\sum_{n=1}^N \|z_n-y_n\|^2 \leq \|e_1\|^2+C\sum_{n=1}^N\theta_n.$$
So, since $\sum_{n\geq 1}\theta_n<+\infty$, \eqref{e22} yields
$$ \sum_{n=1}^{+\infty} \|z_n-y_n\|^2 < +\infty. $$

Consider again an $n\geq 1$. As $ x_{n+1}-u_n = \kappa_n(z_n-y_n)$, one has
$$ \|x_{n+1}-u_n\|^2 \leq \overline\kappa^2 \|z_n-y_n\|^2,$$
so
$$ \sum_{n=1}^{+\infty} \|x_{n+1}-u_n\|^2<+\infty.$$
But $ x_{n+1}-u_n = d_{n+1}-\theta_n d_n$, so
\begin{equation}\label{e23} \sum_{n=1}^{+\infty} \|d_{n+1}-\theta_n d_n\|^2 < +\infty. \end{equation}
From Young's weighted inequality, one has
$$ \|d_{n+1}\|^2 = \|(d_{n+1}-\theta_n d_n)+\theta_n d_n\|^2,$$
so, denoting $ \eta=\frac{2\bar\theta^2}{1-\bar\theta^2}$, one gets
$$ \|d_{n+1}\|^2 \leq (1+\eta) \|d_{n+1}-\theta_n d_n\|^2 + \left(1+\frac{1}{\eta}\right) \theta_n^2\|d_n\|^2. $$
Denoting $q= (1/2) (1+\bar\theta^2)$, as $\theta_n\leq\bar\theta$, one gets
$$ \left(1+\frac1\eta\right)\theta_n^2 \leq q<1,$$

\begin{equation}\label{e24} \|d_{n+1}\|^2 \leq (1+\eta) \|d_{n+1}-\theta_n d_n\|^2 + q\|d_n\|^2
\end{equation}
Summing in \eqref{e24} from $n=1$ to $N$ gives
$$ (1-q)\sum_{n=1}^N\|d_{n+1}\|^2 \leq (1+\eta) \sum_{n=1}^N \|d_{n+1}-\theta_n d_n\|^2 +\|d_1\|^2,$$
so, letting $N\to+\infty$ and using \eqref{e23}, one has
$$ \sum_{n=1}^{+\infty}\|d_n\|^2 = \sum_{n=1}^{+\infty} \|x_n-x_{n-1}\|^2 < +\infty.$$
\end{proof}

\begin{remark}
The hypotheses of Proposition \ref{pr1} are consistent with the ones of the convergence statements for similar algorithms from the literature; see \cite[Assumption 3.1] {HDO} and \cite[Assumption 2.3]{YJS}. Note also that, for example, the sequence generated by
$$ \theta_n=\frac{c}{(n+1)^{1+\varepsilon}},\ n\geq 0,$$
where  $c\ge 0$ and $\varepsilon>0$ satisfies \eqref{e16}.
\end{remark}

\section{Parametrized inertial Douglas-Rachford DC algorithm}\label{section4}

After providing in Section \ref{section3} an iterative method for solving the generalized DC optimization problem \eqref{en1.3}, we propose in this section a new algorithm which increases the parameters range and can be employed for solving a more general class of problems than \eqref{en1.3}(see \cite{bian2021parameterized,miao2024parameterized}). Before presenting and investigating the algorithm, we first introduce this general class of problems

In this section, we propose a parametrized method for solving problem \eqref{1.1} and study its convergence behavior. For a fixed $\alpha >0$, we consider the inclusion problem
\begin{align}\label{P3.2}
    \text{ find } x \in \mathcal{H} \text{ such that } \nabla h(x) \in \partial \phi (x) + \partial g(x) + (2-\alpha )(x),
\end{align}
whose solution set is denoted by
\begin{align}
    \Omega_{PDRDC}= \{x^*\in \mathcal{H}:\nabla h(x^*) \in \partial \phi (x^*) + \partial g(x^*) + (2-\alpha )(x^*) \}.
\end{align}
We assume $\Omega_{PDRDC}\neq \emptyset$.

\begin{remark}\label{re3}
When $\alpha=2$, \eqref{P3.2} turns into \eqref{en1.3}, so one can see the investigations in this section as a generalization of those in Section \ref{section3}. However, as one can see below, the extension we propose comes at a cost, so not everything in Section \ref{section3} can be recovered as a special case of what we obtain in this section.
\end{remark}

\begin{lemma}\label{L3.3}
    Given $ x^* \in \mathcal{H}$, one has $x^* \in \Omega_{PDRDC}$ if and only if there exists a $y\in \mathcal{H} $ satisfying
    \begin{enumerate}

         \item[(i)]   $ x^* = \prox_{\phi}(y)$

        \item[(ii)] $ x^* =\prox_{ g}(\alpha x^* -y+\nabla h(x^*))$.

    \end{enumerate}

\end{lemma}

\begin{proof}
    By construction, $x^*$ is a solution of \eqref{P3.2} if there exists $y\in \mathcal{H}$ such that
   $$y-x^* \in \partial\phi(x^*)  \text{ and } x^* -y+\nabla h(x^*)\in \partial g(x^*)+(2-\alpha)x^*,$$
   which is further equivalent to
$$
       y \in  x^* + \partial\phi(x^*)  \text{ and } \alpha x^* -y+\nabla h(x^*)\in x^*+\partial g(x^*)=(\Id+  \partial g)(x^*)$$
and, finally, to
\begin{equation}\label{eq66}
\left\{
\begin{array}{l}
x^* = \prox_{ \phi} (y)\\
x^*= \prox_{ g} (\alpha x^* - y +  \nabla h(x^*)).
\end{array}
\right.
\end{equation}
\end{proof}

Thus, finding a solution $x^* \in \Omega_{PDRDC}$ of \eqref{P3.2} amounts to finding a $y\in \mathcal{H}$ satisfying equation \eqref{eq66}. Define the set
$$\mathcal{Z}=\{y\in \mathcal{H}|\exists x^* \in  \Omega_{PDRDC}:x^* = \prox_{ \phi} (y)= \prox_{ g} (\alpha x^* - y +  \nabla h(x^*))\}$$ and, for $x^* \in  \Omega_{PDRDC}$, set
$$\mathcal{Z}(x^*)=\{y\in \mathcal{H}|x^* = \prox_{ \phi} (y)= \prox_{g} (\alpha x^* - y +  \nabla h(x^*))\},$$
which is nonempty via Lemma \ref{L3.3}, since $\Omega_{PDRDC}\neq \emptyset$. One can observe that $\mathcal{Z}=\cup_{x^* \in \Omega_{PDRDC}}\mathcal{Z}(x^*)$.

Next, we propose a parametrized inertial Douglas-Rachford DC Algorithm (PDRDC) for solving problem \eqref{P3.2} that is inspired by the method proposed in \cite[Lemma 3.4]{wang2019parameterized}.

	\begin{algorithm}[H]
		\SetAlgoLined

		\KwIn{\begin{enumerate}
                \item  $\alpha,\beta>0$ are real numbers such that $\alpha \in (1,2)$.
                \item Sequence $\{\theta_n\} \subseteq (0,1)$.
				\item Initial points $x_0 \text{ and } x_1 \in \mathcal{H}$.
			\end{enumerate}	
		}
		\SetKwBlock{For}{For}{}
		\SetKwProg{}{}{}{}
		\For($n\in \mathbb{N}$; ){
			$u_n=x_n+\theta_n(x_n-x_{n-1})$\\
     $y_n= \prox_{ \phi} (u_n)$\\
     $z_n=\prox_{g}(\alpha y_n- u_n+\nabla h(y_n))$\\
     $x_{n+1}= u_n+(z_n-y_n)$.
		}
		\KwOut{$x_{n+1},y_n \text{ and } z_n$.}		
		\caption{Parametrized Inertial Douglas-Rachford DC Algorithm (PDRDC) }\label{P4Al4.1}
	\end{algorithm}

Algorithm \ref{P4Al4.1} offers a wider range of parameter choices than Algorithm \ref{P4Al4.11} by making use of the term $\alpha$. The convergence analysis of Algorithm \ref{P4Al4.11} can be summarized in the following results.
The first of them is a technical statement similar to Lemma \ref{lem3.1}, adapted for Algorithm \ref{P4Al4.1}.

\begin{lemma}\label{lem4.2}
    Consider $w \in \Omega_{PDRDC}$ and $y \in \mathcal{Z}(w)$, then sequences $\{u_n\}_n \text{ and } \{y_n\}$ generated by Algorithm \ref{P4Al4.1} satisfy the following inequality
    \begin{equation}\label{eq23}
        (1+\rho)\|y_n -w\|^2 \leq \langle y_n-w,u_n-y \rangle\ \forall n \in \mathbb{N}
    \end{equation}
\end{lemma}
\begin{proof}
    Since $w \in \Omega_{PDRDC}$, from  \eqref{eq66}, we have
    \begin{align}\label{eqn24}
       & w=\prox_{\phi}(y) \iff (\Id+\partial \phi)(w) = y \iff y-w \in \partial \phi (w).
    \end{align}
   Let $n\in \mathbb{N}$. From Algorithm \ref{P4Al4.1}, we have
   \begin{align}\label{eqn25}
    &y_n =\prox_{\phi}(u_n) \iff y_n + \partial \phi (y_n) = u_n \iff u_n-y_n \in \partial{\phi} (y_n).
   \end{align}

Since $\phi$ is $\rho$-strongly convex, we have
\begin{align*}
    \rho\|y_n -w\|^2 \leq \langle y_n -w, \theta -\alpha \rangle \ \forall \theta \in \partial (y_n) \ \forall  \alpha \in \partial (w).
     \end{align*}
From equations \eqref{eqn24} and \eqref{eqn25}, one has
\begin{align}
    \rho\|y_n -w\|^2 &\leq \langle y_n-w, u_n-y_n-y+w \rangle\nonumber\\
    &= \langle y_n-w, u_n-y \rangle - \| y_n-w \|^2, \nonumber\end{align}
   which yields \eqref{eq23}.
\end{proof}

\begin{theorem}
Let $\{x_n\}_n$, $\{u_n\}_n$, $\{y_n\}_n$, $\{z_n\}_n$ be sequences generated by Algorithm \ref{P4Al4.1}. If $\sum_{n=1}^{+\infty}\|x_n-x_{n-1}\|^2 < +\infty$,
then $\{x_n\}_n$ and $\{u_n\}_n$ converge weakly to the same point in $\mathcal{Z}$. Moreover, the sequences $\{y_n\}_n$ and $\{z_n\}_n$ converge strongly to the same limit in $\Omega_{PDRDC}$.
\end{theorem}
\begin{proof}
Let $n\geq 1$.
From the definition of Algorithm \ref{P4Al4.1}, we have for a $w\in \Omega_{PDRDC}$ and a $y\in \mathcal{Z}(w)$
    \begin{align}
        \|x_{n+1}-y\|^2&=\|u_n+(z_n-y_n)-y\|^2\nonumber\\
        &=\left\| \left(1-\frac{1}{\alpha}\right)(u_n-y)+\frac{1}{\alpha}(\alpha z_n-\alpha y_n+u_n-y)\right\|^2\nonumber\\
        &=\left(1-\frac{1}{\alpha}\right)\|u_n-y\|^2+\frac{1}{\alpha}\|\alpha z_n-\alpha y_n+u_n-y\|^2
        \nonumber\\
       & -\frac{1}{\alpha}\left(1-\frac{1}{\alpha}\right)\|\alpha z_n-\alpha y_n\|^2\label{eqn18}.
    \end{align}
Like in the proof of Lemma \ref{lem3.1}, one has (from the definition of Algorithm \ref{P4Al4.1})
\begin{equation}\label{e3.19}
    u_n-y_n\in \partial \phi(y_n)\end{equation}
    \begin{equation}\label{e3.20}
    \alpha y_n - u_n + \nabla h(y_n) -z_n \in \partial g(z_n).\end{equation}

   As $w\in \Omega_{PDRDC}$ and $y\in \mathcal{Z}(w)$, it holds (via Lemma \ref{L3.3})
\begin{equation}\label{e3.21}
    \left\{ \begin{array}{l}
  y-w \in \partial \phi(w)\\
  \nabla h(w)-y + (\alpha-1)w \in \partial g(w).
                \end{array}\right.
\end{equation}

    Using the $\rho$-strong convexity of $g$, equations \eqref{e3.20} and \eqref{e3.21} and the Lipschitz-continuity of $\nabla h$, we have
    \begin{align*}
        2\rho \|z_n-w\|^2&\le 2\langle z_n-w, \nabla h(y_n)-\nabla h(w)+\alpha y_n-z_n-u_n-(\alpha-1)w+y\rangle \\
        &=2\langle z_n-w, \nabla h(y_n)-\nabla h(w) \rangle \\
        &~~+ 2\langle z_n-w,\alpha y_n-\alpha z_n+y-u_n+(\alpha-1)(z_n-w)\rangle \\
        &= 2\langle z_n-w, \nabla h(y_n)-\nabla h(w) \rangle \\
        &~~+\frac{2}{\alpha}\langle \alpha (z_n-w), \alpha y_n+y-u_n+(1-\alpha)w-z_n\rangle \\
        &= 2\langle z_n-w, \nabla h(y_n)-\nabla h(w) \rangle \\
        &~~+\frac{2}{\alpha}\langle \alpha (z_n-w), \alpha y_n+y-u_n-\alpha w\rangle -\frac{2}{\alpha}\langle \alpha (z_n-w), z_n-w\rangle \\
        &\leq L\|z_n-w\|^2+L\|y_n-w\|^2-\frac{1}{\alpha}\left\{ \|\alpha y_n-\alpha z_n+y-u_n\|^2\right\} \\
        &~~+\frac{1}{\alpha}\left\{ \|\alpha(z_n-w)\|^2+\|\alpha y_n+y-u_n-\alpha w\|^2\right\}-2\|z_n-w\|^2,\\
           \end{align*}
    which implies that
    \begin{align}
        \frac{1}{\alpha}&\|\alpha y_n-\alpha z_n+y-u_n\|^2\nonumber\\
        &\le L\|z_n-w\|^2+L\|y_n-w\|^2 +\alpha \|z_n-w\|^2+\frac{1}{\alpha}\|\alpha y_n-\alpha w+y-u_n\|^2\nonumber\\
        &~~-2\|z_n-w\|^2-2\rho \|z_n-w\|^2\nonumber\\
        &= (\alpha +L -2-2\rho )\|z_n-w\|^2+\frac{1}{\alpha}\|\alpha y_n+y-u_n-\alpha w\|^2+L\|y_n-w\|^2.\label{3.13}
    \end{align}
    Now, also employing Lemma \ref{lem4.2}, one has
    \begin{align}
        \|\alpha y_n+y-u_n-\alpha w\|^2
        &~~=\|\alpha (y_n-w)-(u_n-y)\|^2\nonumber\\
        &~~=\alpha^2\|y_n-w\|^2+\|u_n-y\|^2-2\alpha \langle y_n-w, u_n-y\rangle \nonumber\\
        &~~\leq \alpha^2\|y_n-w\|^2+\|u_n-y\|^2-2\alpha (1+\rho)\|y_n-w\|^2.\label{3.14}
    \end{align}

  From \eqref{3.13} and \eqref{3.14}, we have
  \begin{align}
      \frac{1}{\alpha}\|\alpha y_n-\alpha z_n+y-u_n\|^2&\le (\alpha +L -2-2\rho )\|z_n-w\|^2+\alpha\|y_n-w\|^2\nonumber\\
      &~~+\frac{1}{\alpha}\|u_n-y\|^2-2(1+\rho)\|y_n-w\|^2+L\|y_n-w\|^2.\nonumber
  \end{align}
  Now, plugging this value in \eqref{eqn18}, we have
  \begin{align}
      \|x_{n+1}-y\|^2&\le \left(1-\frac{1}{\alpha}\right)\|u_n-y\|^2+(\alpha+L-2-2\rho)\|z_n-w\|^2+\alpha \|y_n-w\|^2\nonumber\\
      &~~+\frac{1}{\alpha}\|u_n-y\|^2-2(1+\rho)\|y_n-w\|^2+L\|y_n-w\|^2-(\alpha-1)\|z_n- y_n\|^2\nonumber\\
      &=\|u_n-y\|^2 +(\alpha+L-2-2\rho)(\|z_n-w\|^2+ \|y_n-w\|^2)\label{21}\\
      &~~-\left(\alpha -1\right)\| z_n- y_n\|^2\nonumber\\
      &\le\|(1+\theta_n)(x_n-y)-\theta_n(x_{n-1}-y)\|^2\nonumber\\
      &~~+(\alpha+L-2-2\rho)(\|z_n-w\|^2+ \|y_n-w\|^2)\nonumber\\
      &= (1+\theta_n)\|x_n-y\|^2-\theta_n\|x_{n-1}-y\|^2+\theta_n(1+\theta_n)\|x_n-x_{n-1}\|^2\nonumber\\
      &~~+(\alpha+L-2-2\rho)(\|z_n-w\|^2+ \|y_n-w\|^2)\nonumber\\
      &=\|x_n-y\|^2+\theta_n\{\|x_n-y\|^2-\|x_{n-1}-y\|^2\}+\theta_n(1+\theta_n)\|x_n-x_{n-1}\|^2\nonumber\\
      &~~+(\alpha+L-2-2\rho)(\|z_n-w\|^2+ \|y_n-w\|^2).\label{eq7}\\
     &\leq\|x_n-y\|^2+\theta_n\{\|x_n-y\|^2-\|x_{n-1}-y\|^2\}+\theta_n(1+\theta_n)\|x_n-x_{n-1}\|^2,\nonumber\\    \end{align}
 as $\alpha <2$ and $2\rho >L$. Using Lemma \ref{lem2.1}, the sequence $\{\|x_{n+1}-y\|\}_n$ converges. Similarly, one obtains that $\{\|u_{n}-y\|\}_n$ converges, too.

Due to the boundedness of the sequence $\{x_{n}\}_n$,   there exists a subsequence $\{x_{n_j}\}_j$ of $\{x_n\}_n$     such that $x_{n_j}\rightharpoonup \bar x$
    for some $\bar x\in\mathcal{H}$. By taking into account the way $\{u_n\}_n$ is defined, one can derive $u_{n_j}\rightharpoonup \bar x$, too.

Further, taking into account that $\|x_n-x_{n-1}\|\to 0$ (otherwise, the hypothesis $\sum_{n=1}^{+\infty}\|x_n-x_{n-1}\|^2 < +\infty$ would be violated), $L<2\rho$ and $\alpha <2$, \eqref{eq7} yields the convergence to zero of $\{\| z_n-w\|\}_n$ and $\{\| y_n-w\|\}_n$. Hence, $\{y_n\}_n$ and $\{z_n\}_n$ converge strongly toward the same element $w$ of $\Omega_{PDRDC}$.

Also, we have (via \eqref{e3.19} and \eqref{e3.20}) for all $j\in \mathbb{N}$
  $$  u_{n_j}-y_{n_j}\in \partial \phi(y_{n_j})$$
    and
   $$\nabla h(y_{n_j})+\alpha y_{n_j}-u_{n_j}-z_{n_j}\in \partial g (z_{n_j}).$$
Employing \cite[Proposition 16.36]{BAC}, one obtains
     $\bar{x}-w\in \partial \phi (w)$, which leads to $w= \prox_{\phi}(\bar{x})$ as well as
    $\nabla h(w)+w-\bar{x}\in \partial g(w)$, which yields
     $w= \prox_{ g}(2w-\bar{x}+ \nabla h(w))$,
  so $\bar x \in \mathcal{Z} (w)$. The lemma \ref{lem2.2} then yields $x_{n}\rightharpoonup \bar x\in \mathcal{Z}(w) \subseteq \mathcal{Z}$, and the same conclusion can be drawn for $\{u_n\}_n$.

\end{proof}

Next we discuss about hypotheses that guarantee $\sum_{n=1}^{+\infty}\|x_n-x_{n-1}\|^2 < +\infty$ for the sequence
$\{x_n\}_n$ generated by Algorithm \ref{P4Al4.1}.

\begin{proposition}\label{pr2}
Let $\{x_n\}_n$, $\{u_n\}_n$, $\{y_n\}_n$ and $\{z_n\}_n$ be sequences generated by Algorithm \ref{P4Al4.1}, and assume that

\begin{equation}\label{e16.1}
 \bar\theta = \sup_{n\ge 1} \theta_n <1,\ \Theta=\sum_{n=1}^{+\infty}\theta_n<+\infty.
\end{equation}
Then $\sum_{n=1}^{+\infty}\|x_n-x_{n-1}\|^2 < +\infty$.
\end{proposition}

\begin{proof}
  From inequality \eqref{21}, one can find
  \begin{align}\label{eqn4.43}
          \|x_{n+1}-y\|^2 &\leq \|u_n-y\|^2 +(\alpha+L-2-2\rho)(\|z_n-w\|^2+ \|y_n-w\|^2)\\
&~~-\left(\alpha -1\right)\| z_n- y_n\|^2.\nonumber
    \end{align}
    Denoting $$b_{\rho}=(2-\alpha+ 2 \rho-L)$$
    $$\alpha_k=\alpha -1,$$
    we can rewrite \eqref{eqn4.43} as
      \begin{align}\label{eqn4.431}
          \|x_{n+1}-y\|^2 &\leq \|u_n-y\|^2 -b_{\rho}(\|z_n-w\|^2+ \|y_n-w\|^2)-\alpha_k\| z_n- y_n\|^2 \nonumber.
    \end{align}
    For $n\geq 1$, consider $d_n=x_n-x_{n-1}, e_n =x_n -y, u_n =x_n +\theta_n d_n \text{ and } u_n-y=e_n+\theta_n d_n$ so,

 \begin{align}
      \|u_n-y\|^2 &= \|x_n+\theta_n (x_n -x_{n-1})-y)\|^2\nonumber\\
       &= \|e_n +\theta_n d_n\|^2 \nonumber\\
       &=\|e_n\|^2+\theta_n^2 \|d_n\|^2+2\theta_n\langle e_n, d_n\rangle.\nonumber
 \end{align}
  Since, $$2\langle e_n, d_n \rangle \leq \|e_n\|^2 + \|d_n\|^2,$$
  we get,
  \begin{align}
      \|u_n -y\|^2 &\leq \|e_n\|^2 + \theta_n^2 \|d_n\|^2 +\theta_n \|e_n\|^2+\theta_n \|d_n\|^2\nonumber\\
      &= (1+\theta_n)\|e_n\|^2 +\theta_n (1+\theta_n)\|d_n\|^2. \nonumber
  \end{align}
  From \eqref{eqn4.43} we obtain,
  \begin{align}
      \|e_{n+1}\|^2 &\leq (1+\theta_n)\|e_n\|^2 +\theta_n(1+\theta_n)\|d_n\|^2-b_{\rho} (\|z_n-w\|^2+\|y_n-w\|^2)-\alpha_k \|z_n-y_n\|^2. \nonumber
  \end{align}
  Since $d_n=e_n-e_{n-1},$ we have, $$\|d_n\|^2\leq 2\|e_n\|^2+2\|e_{n-1}\|^2$$
  which inforce
  \begin{align}
      \|e_{n+1}\|^2 &\leq (1+\theta_n)\|e_n\|^2 +\theta_n (1+\theta_n)\|d_n\|^2\nonumber\\
      &=(1+\theta_n)\|e_n\|^2+\theta_n (1+\theta_n)\{2\|e_n^2 +2\|e_{n-1}\|^2\}\nonumber\\
      &= (1+\theta_n+2\theta_n+2\theta_n^2)\|e_n\|^2+(2\theta_n +2\theta_n^2)\|e_{n-1}\|^2. \nonumber
  \end{align}
Denote,
      $$M_n=\max\{\|e_n\|^2, \|e_{n-1}\|^2\},$$ we obtain
      $$M_{n+1}\leq (1+5\theta_n+4 \theta_n^2)M_n \leq (1+\theta_n)M_n,$$
where we took into account $$\|e_n\|^2\leq M_{n+1}\leq (1+9\theta_n)M_n.$$
Then, $M_{n+1}\leq M_1 \prod_{k=1}^{n}(1+9\theta_k).$
Using the inequality, $1+t\leq e^t, \forall t\geq 0,$ one gets
$$\prod_{k=1}^{n}(1+9\theta_k)\leq \prod_{k=1}^{n} e^{9 \theta_k}=e^{9\sum_{k=1}^{n}\theta_k}.$$
So, $$M_{n+1}\leq M_1 e^{9\sum_{k=1}^{n}\theta_k},$$
which implies $$M_{n+1} \leq M_1 e^{9\Theta}.$$
So, $$\sup_{n\geq 1}M_n < +\infty.$$
Noting that $$M_n =\max\{\|x_n-y\|^2,\|x_{n-1}-y\|^2\}$$
one obtains then $$\sup_{n \geq 1} \|x_n-y\| < +\infty.$$
So cosequently, $$\|d_n\|=\|x_n -x_{n-1}\|\leq \|x_n-y\|+\|x_{n-1}-y\|$$
one obtains then $$\sup_{n\geq 1}\|x_n-y\|< +\infty.$$
So consequently, $$\|d_n\|=\|x_n-x_{n-1}\|\leq \|x_n-y\|+\|x_{n-1}-y\|$$
so, $$\sup_{n\geq 1}\|d_n\| <+\infty,$$
thus, $\{d_n\}$ is bounded. Then there exists a $C >0$ such that $$(\theta_n+\theta_n^2)\|d_n\|^2 \leq C\theta_n.$$
Therefore, $$b_p(\|y_n-\bar{x}\|^2+\|z_n -\bar{x}\|^2)+\alpha_k \|z_n-y_n\|^2\le \|e_n\|^2-\|e_{n+1}\|^2+C\theta_n,$$
summing from $n=1$ to an $N>0$ one gets,
\begin{equation}\label{eq4.46}
 b_p\sum_{n=1}^{N}(\|y_n-\bar{x}\|^2+\|z_n -\bar{x}\|^2)+ \alpha_k \sum_{n=1}^{N}\|z_n-y_n\|^2\leq \|e_1\|^2+C\sum_{n=1}^{N} \theta_n.
\end{equation}
Since $\sum_{n\geq 1}\theta_n < +\infty,$ so \eqref{eq4.46} yields $$\sum_{n=1}^{\infty}\|z_n-y_n\|^2 <+\infty.$$
Consider again an $n\geq 1.$ As $x_{n+1}-u_n=z_n-y_n,$ one has
$$\|x_{n+1}-u_n\|^2=\|z_n-y_n\|^2$$
so, $$\sum_{n=1}^{\infty}\|x_{n+1}-u_n\|^2 <+\infty.$$
But, $x_{n+1}-u_n=d_{n+1}-\theta_n d_n $ so,
\begin{equation}\label{eqn4.48}
 \sum_{n=1}^{\infty}\|d_{n+1}-\theta_n d_n\|^2< +\infty.
\end{equation}
From Young's weighted inequality, one has
$$\|d_{n+1}\|^2=\|(d_{n+1}-\theta_n d_n)+\theta_n d_n\|^2$$
so, denoting $\eta= \frac{2\bar{\theta}^2}{1-\bar{\theta}^2},$ one gets
$$\|d_{n+1}\|^2 \leq (1+\eta)\|d_{n+1}-\theta_n d_n\|^2+(1+\frac{1}{\eta})\theta_n^2 \|d_n\|^2.$$
Denoting $q=\frac{1}{2}(1+\bar{\theta}^2)$ as $\theta_n \leq \bar{\theta},$ one gets
$$(1+\frac{1}{\eta})\theta_n^2 \leq q <1,$$
\begin{equation}\label{eqn4.47}
    \|d_{n+1}\|^2\leq (1+\eta)\|d_{n+1}-\theta_nd_n\|^2+q\|d_n\|^2.
\end{equation}
Summing in \eqref{eqn4.47} from $n=1$ to $N$ gives
$$(1-q)\sum_{n=1}^{N}\|d_{n+1}\|^2 \leq(1+\eta)\sum_{n=1}^{N}\|d_{n+1}-\theta_n d_n\|^2+\|d_1\|^2.$$
So, letting $N\to +\infty$ and using \eqref{eqn4.48}, one has
$$\sum_{n=1}^{+\infty}\|d_n\|^2=\sum_{n=1}^{+\infty}\|x_n -x_{n-1}\|^2 < +\infty.$$
\end{proof}
\section{Numerical Experiments}\label{section5}
In this section, we conduct numerical experiments to illustrate the performance of the proposed algorithms in solving the sparse recovery problem. We have also compared the proposed algorithms with \cite[Algorithm 1]{YJS}, and the computational results reveal their superior performance. All the numerical experiments have been performed in \textsc{MATLAB} 2024a-Win64 on a laptop with Intel(R) Core(TM) i5-7200U CPU @ 2.50GHz  2.70 GHz with 8 GB RAM and 64-bit operating system.


\subsection{Experiment 1}
For the numerical experiment, we have considered the $\ell^1$-$\ell^2$ sparse recovery problem \cite{chuang2022unified}
\begin{equation}
    \min_{x\in \mathbb{R}^k} \left\{\frac{1}{2}\|Ax-b\|^2+\lambda(\|x\|_1 -\|x\|_2)\right\}.
\end{equation}

 We have constructed the matrix $A \in \mathbb{R}^{m\times k}$ by using (i) random Gaussian, and (ii) partial DCT matrices satisfying RIP conditions with high probability \cite{candes2006stable}. Throughout the experiment, we have chosen $\theta_n=(n-1)/(2n+2.5)$ $n\ge 1$, $\kappa=0.3$ and $x_0=x_1=0_k$ as initial iterations.
 To have a stopping criterion relevant regardless of the scale of the employed data, we used $\|x_n - x_{n-1}\|/{\max{\{x_n,x_{n-1}\}}} < 10^{-6}$.

In the first experiment, we consider a matrix $A$ of order $64 \times 256$ with sparsity 10. We plotted the convergence behavior of the Algorithm \ref{P4Al4.11} and compared it with \cite[Algorithm 1]{YJS} for $\beta \in \{10^{-1},10^{-2}\}$. In Figure \ref{Fig6.3}, we plotted the graph representing the value of $\|A\cdot-b\|$ depending on the number of iterations of the employed method, and one can see that $\|A\cdot-b\|$ is smaller for Algorithm \ref{P4Al4.11} than when using \cite[Algorithm 1]{YJS} for both $\beta=10^{-1}$ and $\beta=10^{-2}$. 

\begin{figure}[h]
    \centering
    \subfloat[]{%
        \resizebox*{6cm}{!}{\includegraphics{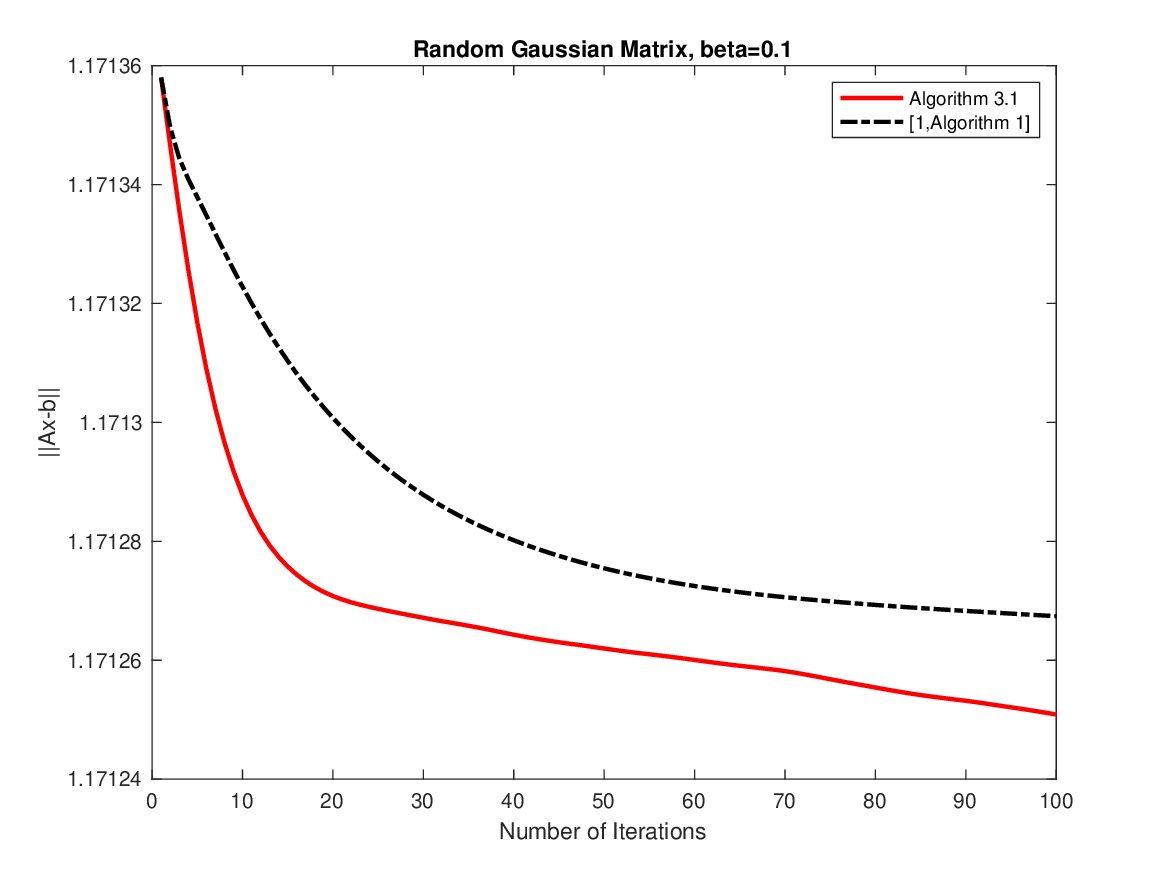}}}\hspace{5pt}
    \label{Fig6.1}
    \subfloat[]{
        \resizebox*{6cm}{!}{\includegraphics{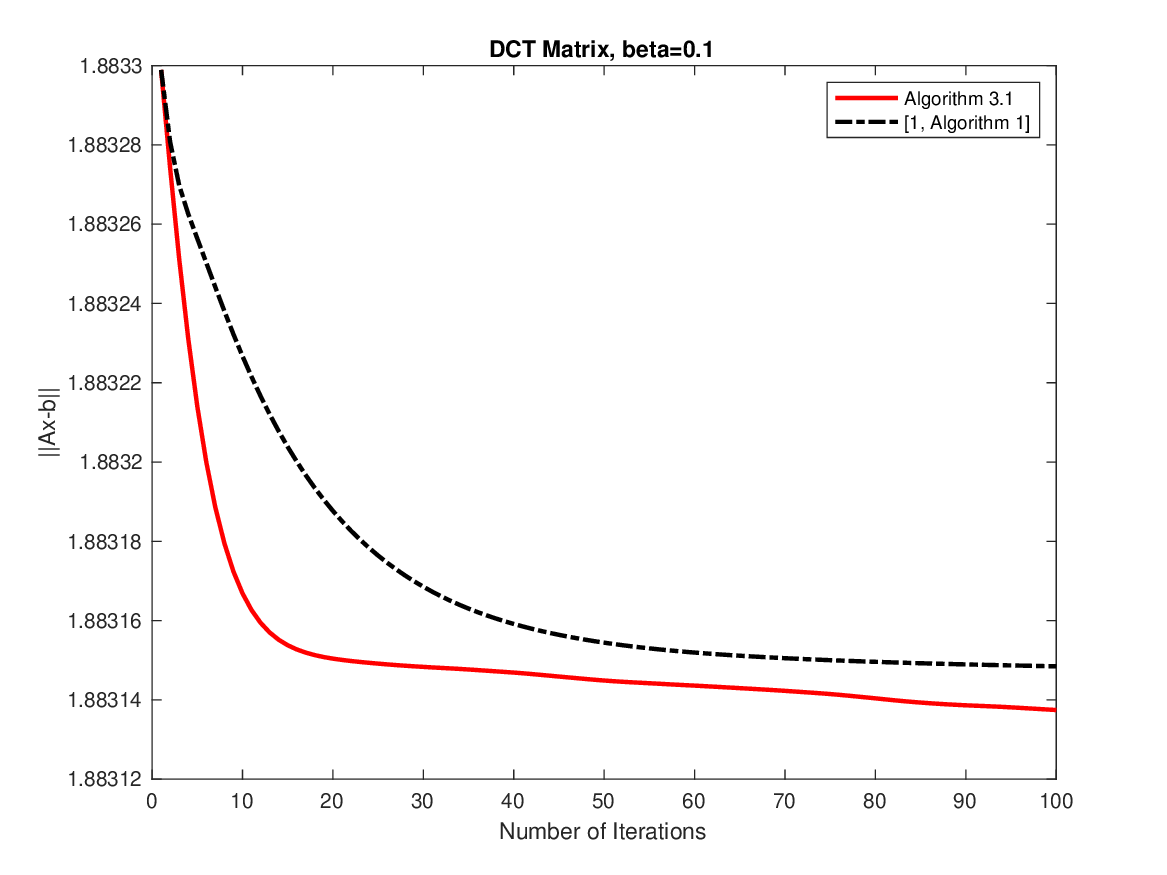}}}
    \label{Fig6.2}
      \subfloat[]{%
        \resizebox*{6cm}{!}{\includegraphics{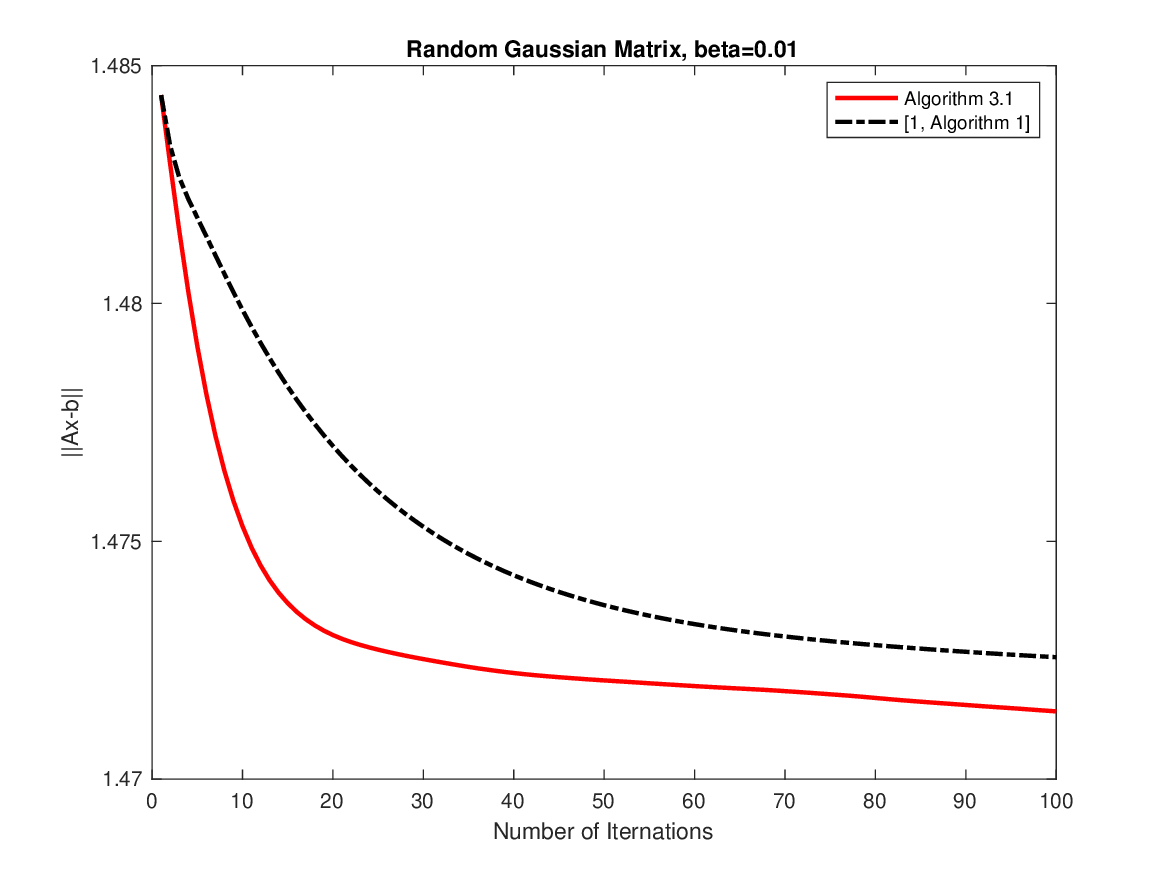}}}\hspace{5pt}
    \label{Fig6.7}
    \subfloat[]{
        \resizebox*{6cm}{!}{\includegraphics{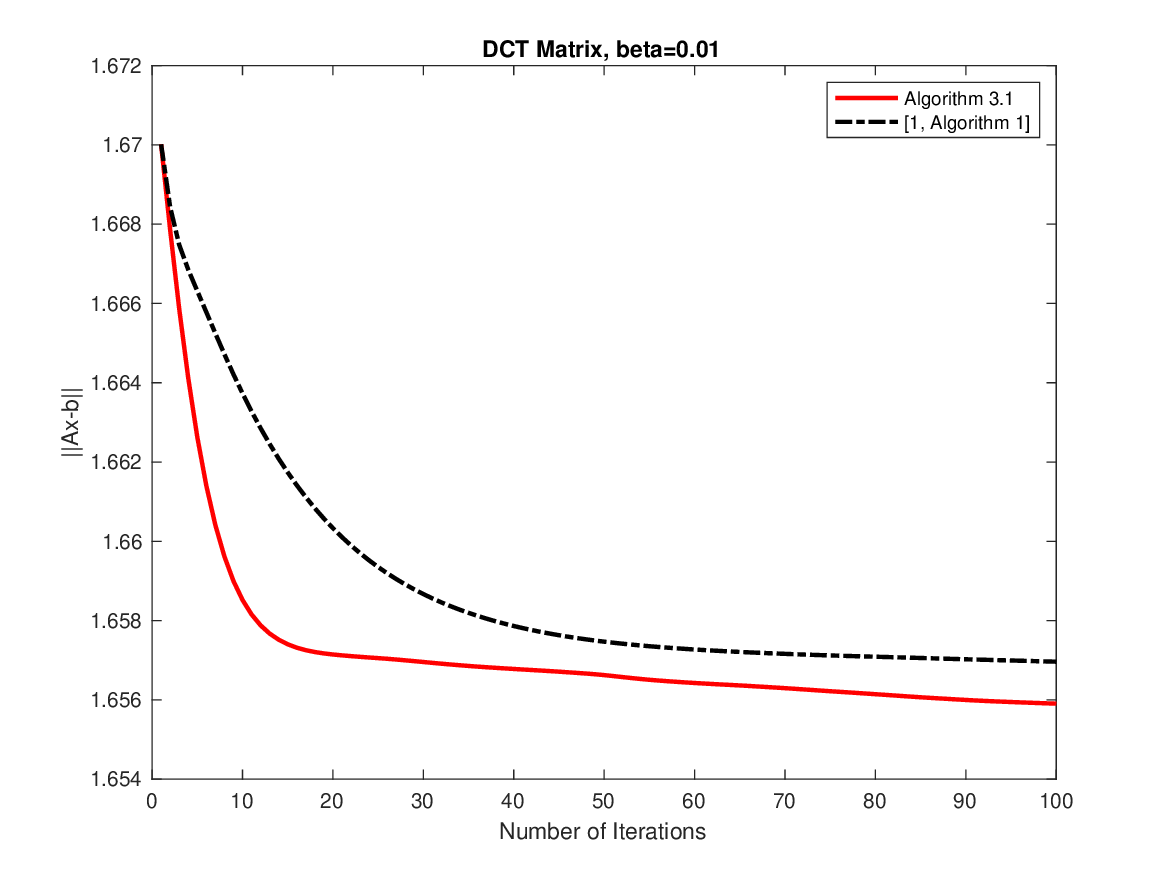}}}
    \label{Fig6.8}
    \caption{The convergence behavior of the proposed Algorithm \ref{P4Al4.11} and of \cite[Algorithm 1]{YJS} for $A\in {\mathbb{R}}^{64}\times \mathbb{R}^{256}$ with sparsity 10.} \label{Fig6.3}
\end{figure}

Next, we consider the matrix $A$ of order $256 \times 1024$ with sparsity $64$. We plotted the graph showing the value of $\|A\cdot-b\|$ depending on the number of iterations, for $\beta =10^{-2}$ and $\beta =10^{-3}$. The convergence behavior of \cite[Algorithm 1]{YJS} and Algorithm \ref{P4Al4.11} is shown in Figure \ref{Fig6.6}. There, one can notice that the value of $\|A\cdot-b\|$ is smaller for Algorithm \ref{P4Al4.11} than for \cite[Algorithm 1]{YJS} for both $\beta=10^{-1} $ and $\beta=10^{-2}$, revealing, again, that Algorithm \ref{P4Al4.11} outperforms \cite[Algorithm 1]{YJS} in both cases even after changing the dimensions of the matrix used.

\begin{figure}[h]
    \centering
    \subfloat[]{%
        \resizebox*{6cm}{!}{\includegraphics{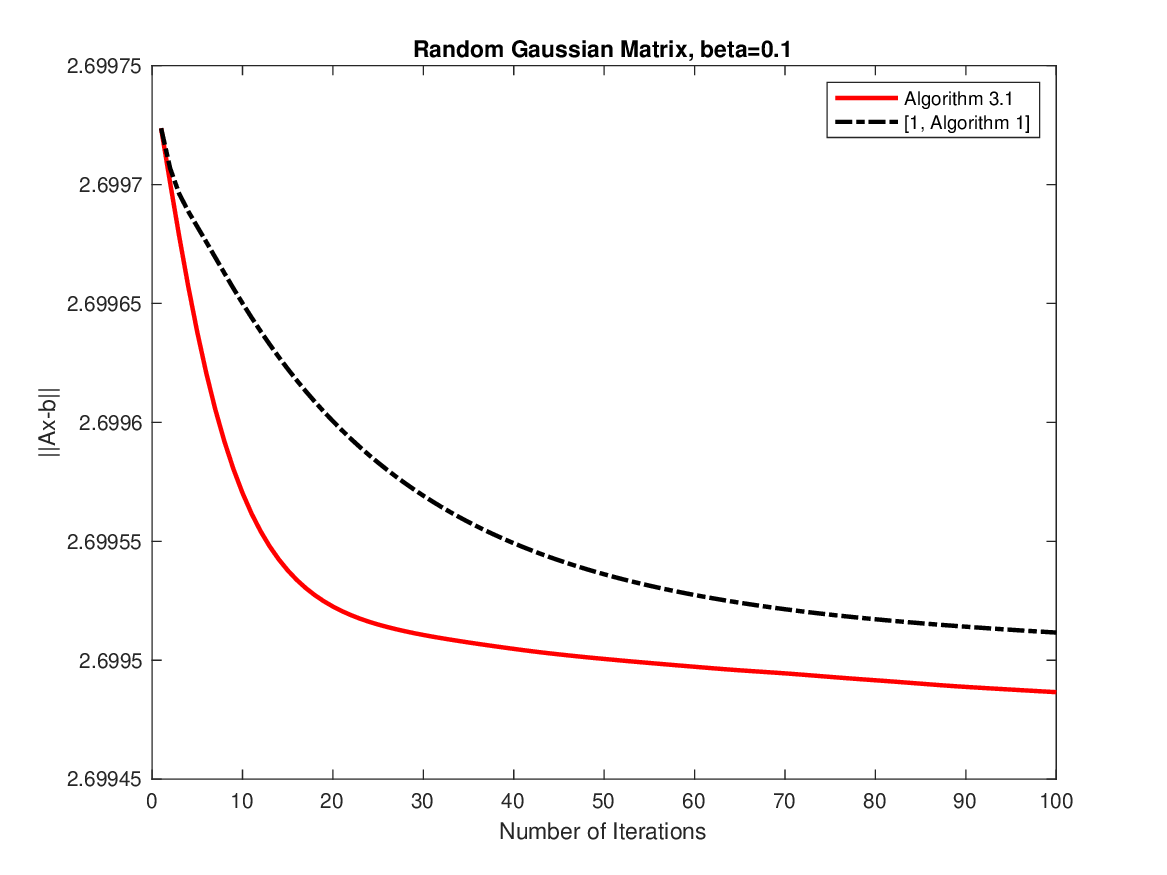}}}\hspace{5pt}
    \label{Fig6.4}
    \subfloat[]{
        \resizebox*{6cm}{!}{\includegraphics{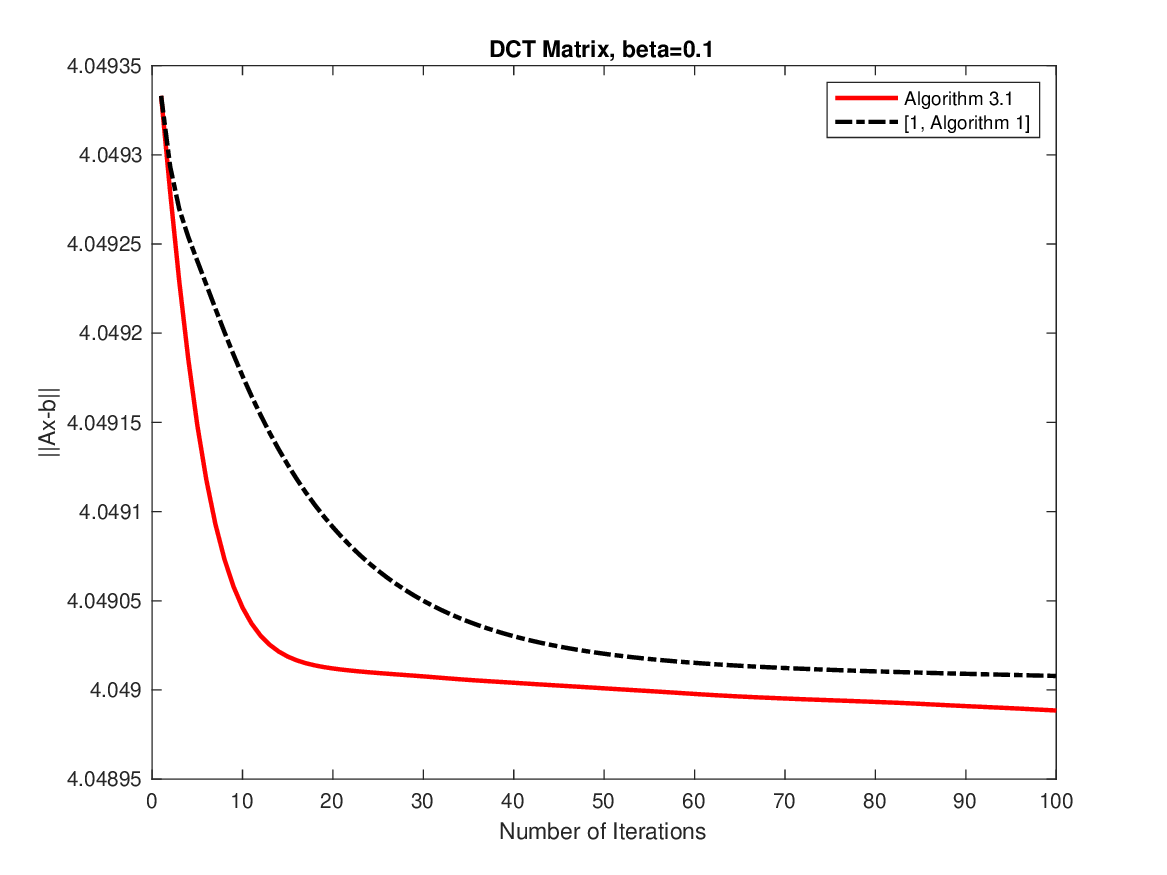}}}
    \label{Fig6.5}
     \subfloat[]{%
        \resizebox*{6cm}{!}{\includegraphics{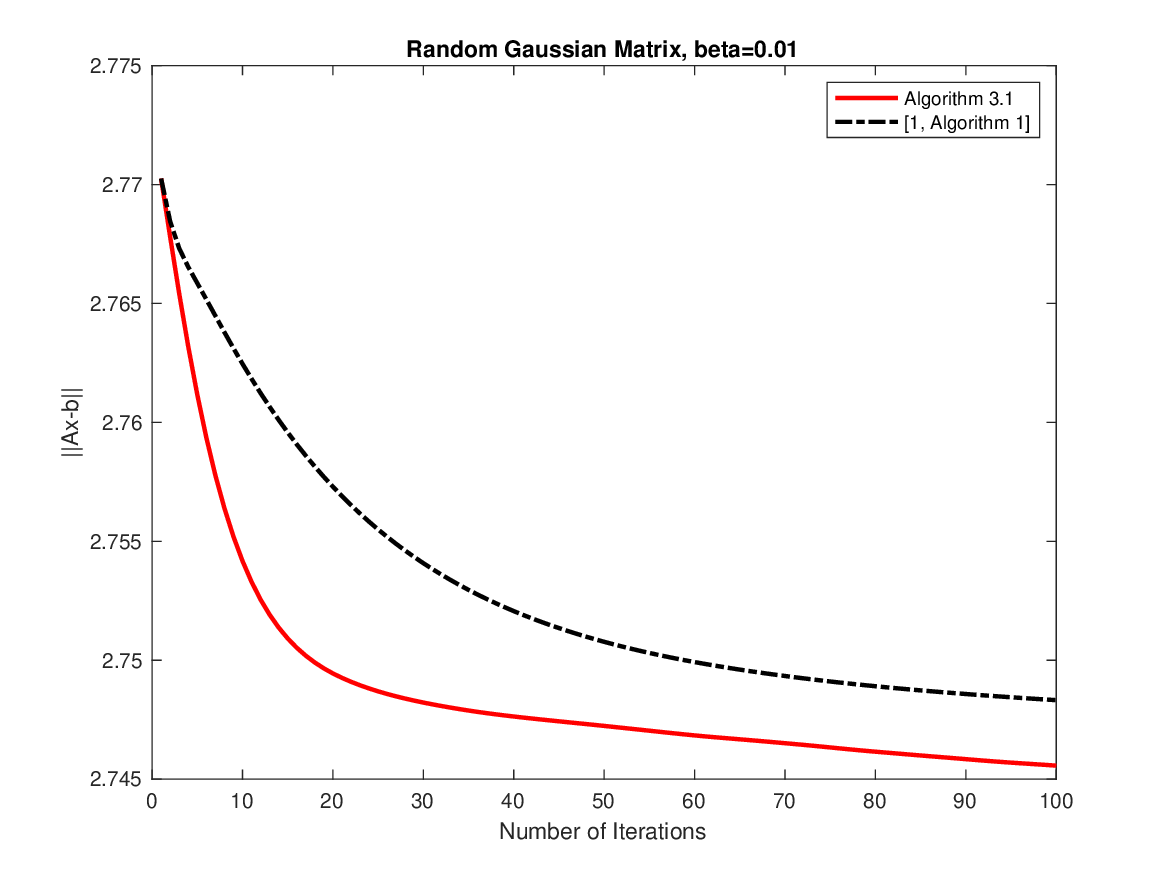}}}\hspace{5pt}
    \label{Fig6.10}
    \subfloat[]{
        \resizebox*{6cm}{!}{\includegraphics{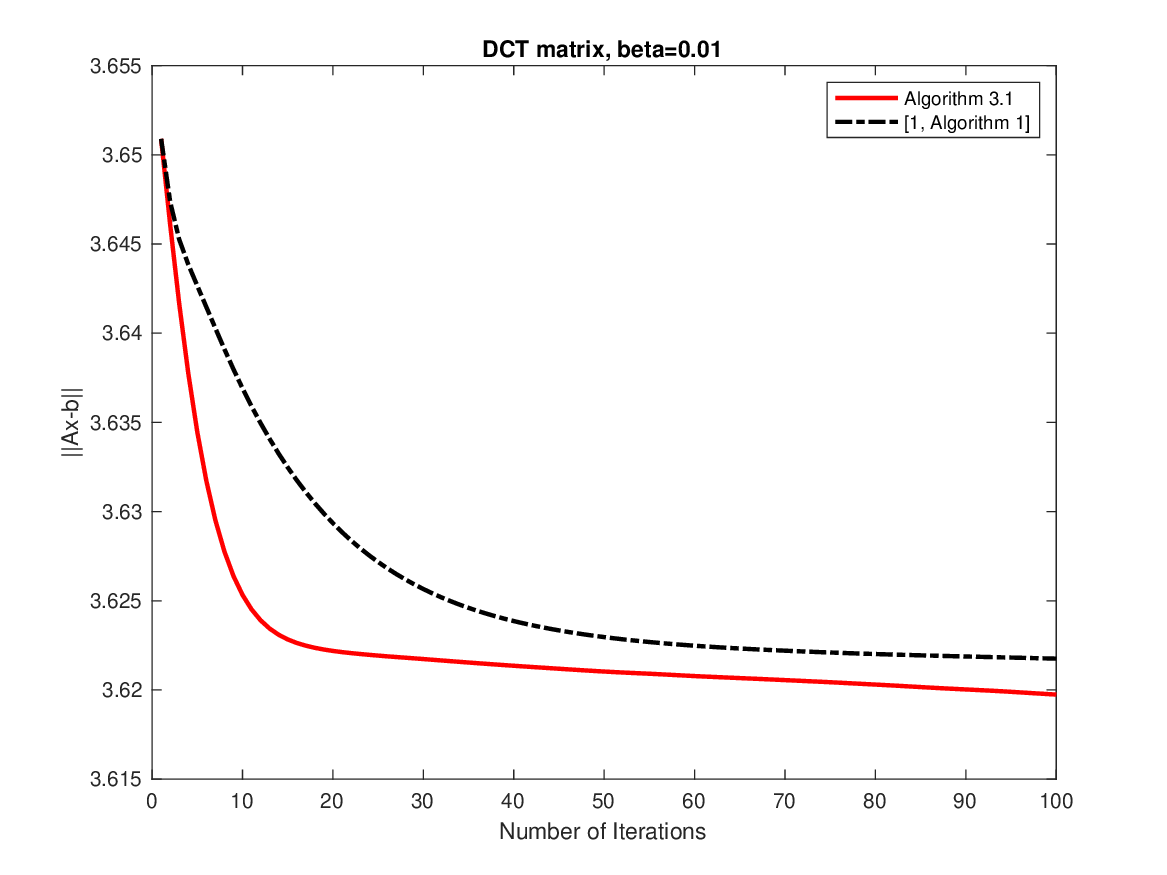}}}
    \label{Fig6.11}
 \caption{The convergence behavior of the proposed Algorithm \ref{P4Al4.11} and of \cite[Algorithm 1]{YJS} for $A\in {\mathbb{R}}^{256}\times \mathbb{R}^{1024}$ with sparsity 64.} \label{Fig6.6}
\end{figure}

In the third experiment, we studied the behavior of the proposed Algorithm \ref{P4Al4.11} with different choices of the proximal parameter $\beta \in \{10^{-1},10^{-2},10^{-3}\}$. The behavior of Algorithm \ref{P4Al4.11} is shown in Figure \ref{Fig6.9}.
One can notice there that the lower the value of $\beta$ is, the higher is the value of $\|A\cdot-b\|$. Even though we have no theoretical result to guarantee this empirical observation, it should be taken into consideration in future implementations of the iterative method we proposed that its convergence speed increases as the proximal parameter $\beta$ decreases.

\begin{figure}[h]
    \centering
    \subfloat[]{%
        \resizebox*{6cm}{!}{\includegraphics{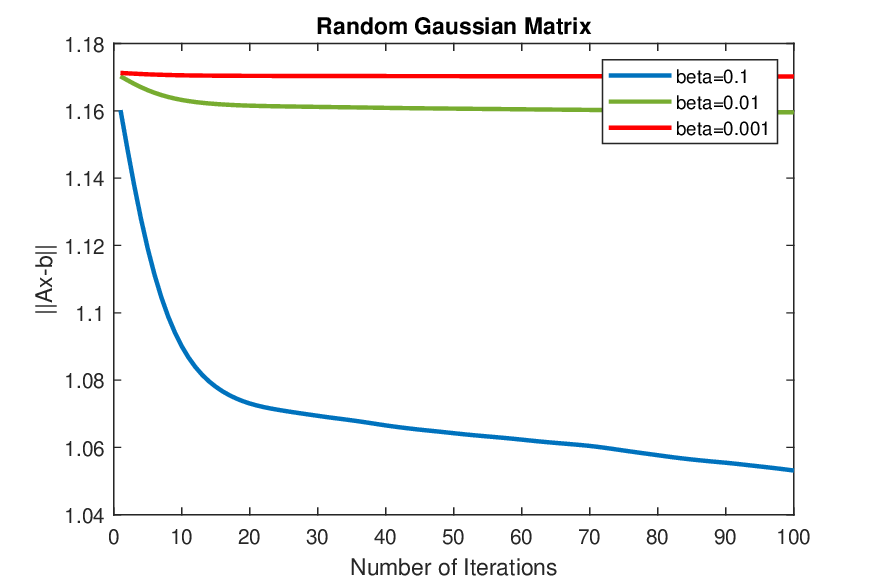}}}\hspace{5pt}
    \label{Fig6.12}
    \subfloat[]{
        \resizebox*{6cm}{!}{\includegraphics{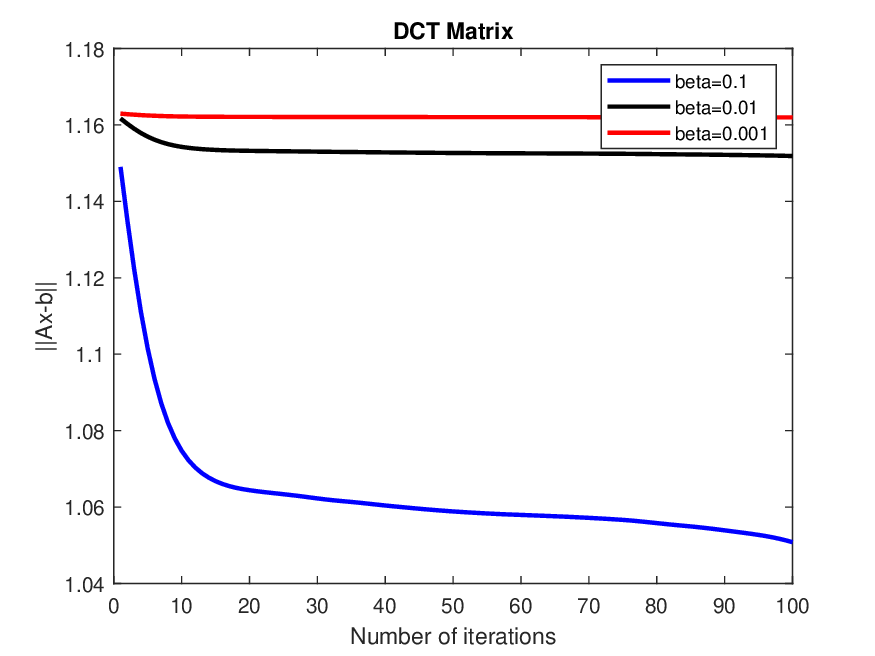}}}\hspace{5pt}
    \label{Fig6.13}
     \subfloat[]{%
        \resizebox*{6cm}{!}{\includegraphics{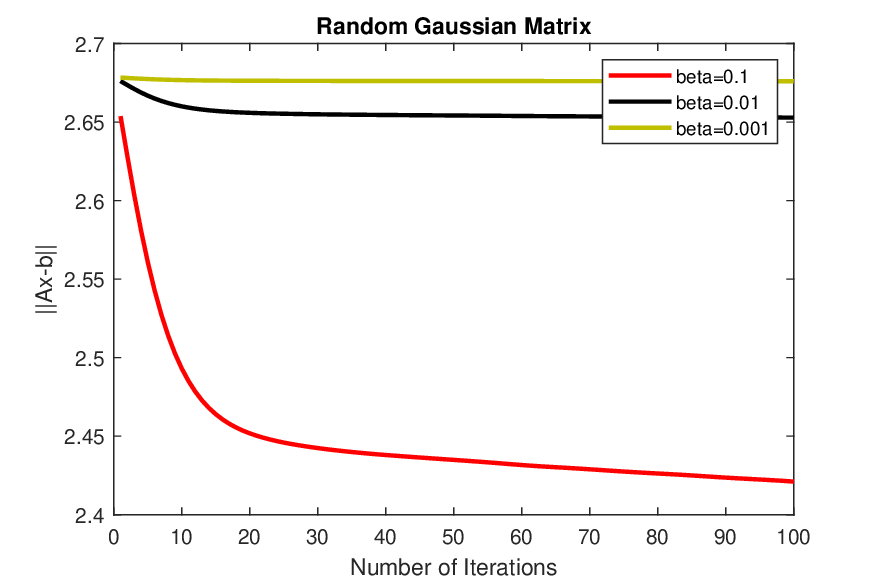}}}\hspace{5pt}
    \label{Fig6.14}
    \subfloat[]{
        \resizebox*{6cm}{!}{\includegraphics{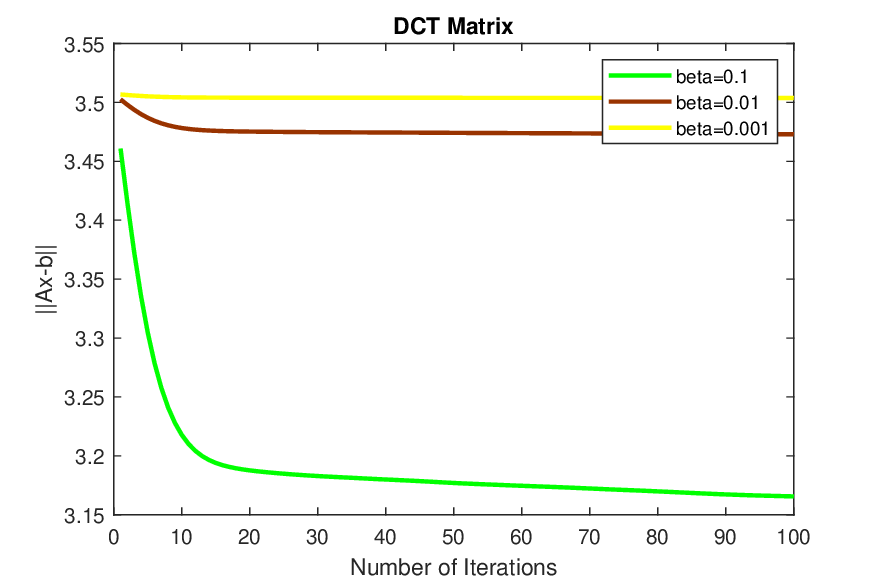}}}
    \label{Fig6.15}
 \caption{The convergence behavior of the proposed Algorithm \ref{P4Al4.11} for different choices of proximal parameter $\beta$.} \label{Fig6.9}
\end{figure}


\subsection{Experiment 2} In this section we have considered a different set-up to test the performance of our proposed algorithm and compare it with an already known algorithm. We consider the following problem:
\begin{align}\label{P5.2}
    \min_{x\in \mathbb{R}^k}\left\{f(x)+g(x)-h(x)\right\},
\end{align}
where $f:\mathbb{R}^n \to (-\infty, +\infty)$ is a function defined by $f(x)=\frac{1}{2}x^TQx+c^Tx+d$ and $Q \in \mathbb{R}^{600\times 400}$ is a positive definite matrix, $c\in \mathbb{R}^k$ is a constant vector and $d$ is a constant scalar, $g:\mathbb{R}^k \to \mathbb{R}$ is defined by $g(x)=\|x\|_1$ and $h:\mathbb{R}^k \to \mathbb{R}$ is the Huber loss function defined by
\begin{align*}
h(x)=\left\{
\begin{array}{l}
      \frac{1}{2}x^2, \|x\|_2 \leq \delta,\\
      \delta(|x|-\frac{1}{2}\delta), \text{ otherwise}.  
\end{array}\right.
  \end{align*}
for a given constant $\delta>0$. Problem \eqref{P5.2} is used for modeling applications arising in machine learning, especially in robust classification with sparse regularization \cite{wu2021nonconvex}. In this experiment, we compared the performance of the proposed Algorithm \ref{P4Al4.11} with \cite[Algorithm 1]{YJS} for different choices of $\beta>0$. We have randomly generated a matrix $Q\in \mathbb{R}^{600\times 400}$ and selected $\kappa=1, \theta=0.9, \theta_n = \frac{n-1}{2n +2.5}, n\in \mathbb{N}$. We have performed the experiment for $1000$ iterations and used $\|x_n -x_{n-1}\|^2 < 10^{-3}$ as a stopping criterion. 

In the first experiment, we have chosen $\delta=0.001$ and ran the algorithms for $\beta \in \{10^{-3},10^{-4},10^{-5}\}$. The performance of Algorithm \ref{P4Al4.11} and \cite[Algorithm 1]{YJS} is shown in Figure \ref{Fig5.2.11}, revealing that the graph of objective values generated by Algorithm \ref{P4Al4.11} is below its counterpart issued by \cite[Algorithm 1]{YJS} for all employed choices of $\beta$. This highlights that the convergence speed of our proposed Algorithm \ref{P4Al4.11} outperforms the one of \cite[Algorithm 1]{YJS}.

\begin{figure}[h]
    \centering
    \subfloat[]{%
        \resizebox*{5.5cm}{!}{\includegraphics{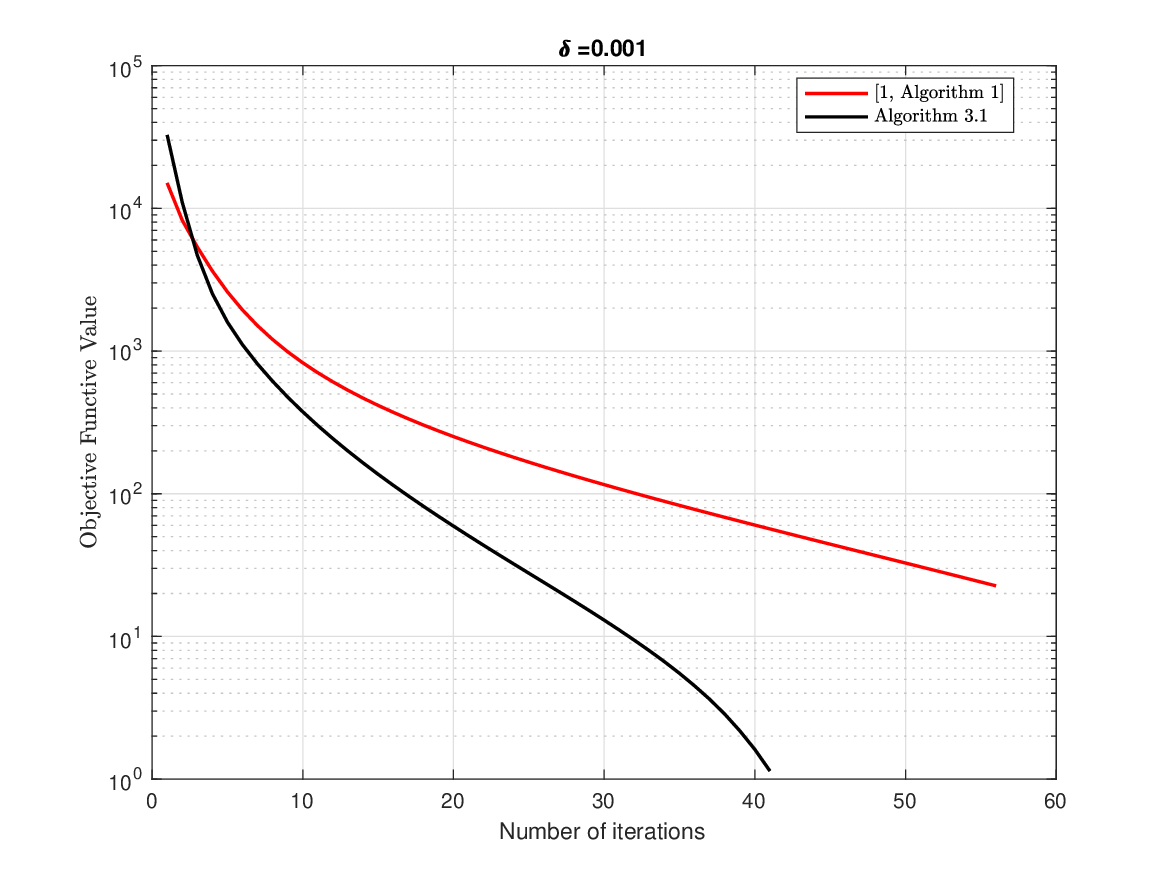}}}\hspace{5pt}
    \label{Fig5.2.1}
    \subfloat[]{
        \resizebox*{5.5cm}{!}{\includegraphics{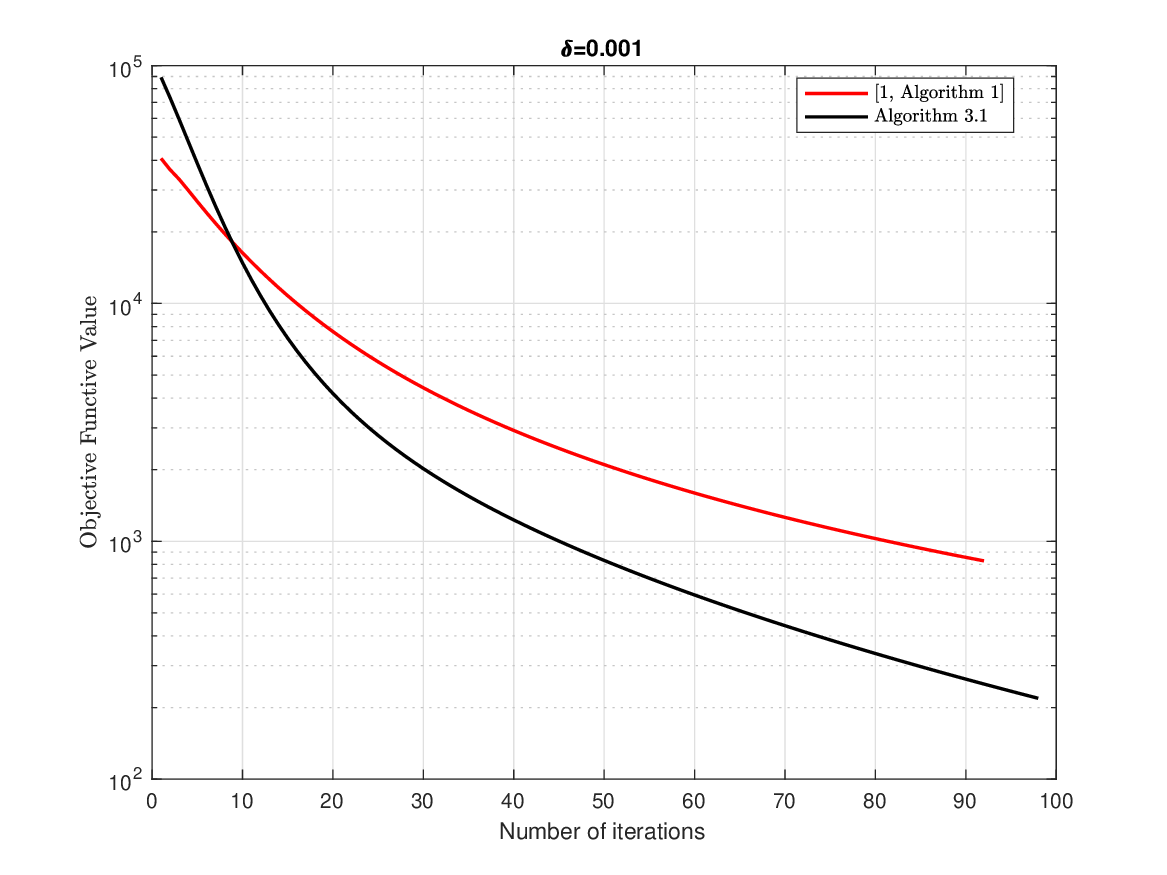}}}
    \label{Fig5.2.2}
     \subfloat[]{%
        \resizebox*{5.5cm}{!}{\includegraphics{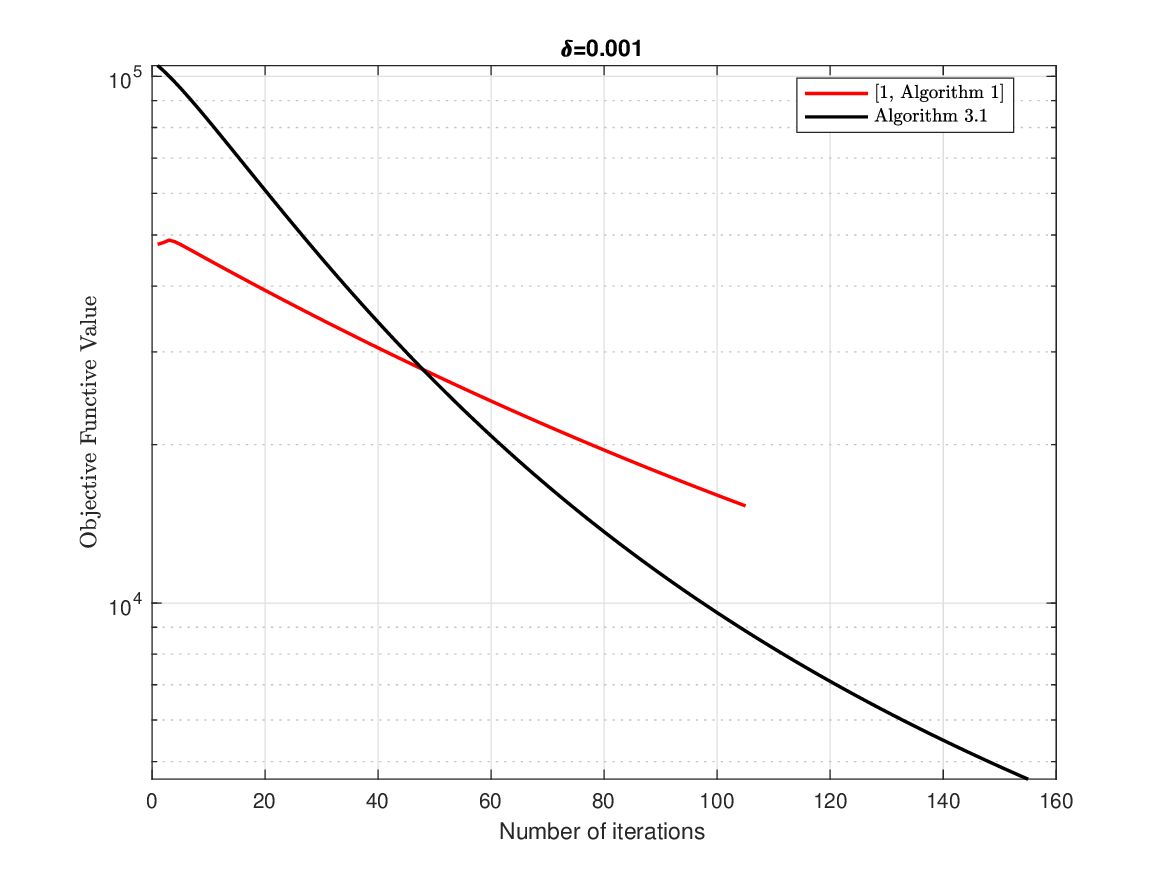}}}\hspace{5pt}
    \label{Fig5.2.3}

 \caption{The convergence behavior of the proposed Algorithm \ref{P4Al4.11} and of \cite[Algorithm 1]{YJS} for $\beta \in \{10^{-3},10^{-4},10^{-5}\}$ and $\delta =0.001$.} \label{Fig5.2.11}
\end{figure}

In second experiment, we have calculated functional value, $F$ for $\beta \in \{10^{-3},10^{-4}, 10^{-5}\}$ and $\delta =0.0001$. Figure \ref{Fig5.2.22}, demonstrates the variation of function value with number of iterations. It shows that graph of Algorithm \ref{P4Al4.11} is below to graph of \cite[Algorithm 1]{YJS}. Thus, Algorithm \ref{P4Al4.11} outperforms \cite[Algorithm 1]{YJS}.

\begin{figure}[h]
    \centering
    \subfloat[]{%
        \resizebox*{5.5cm}{!}{\includegraphics{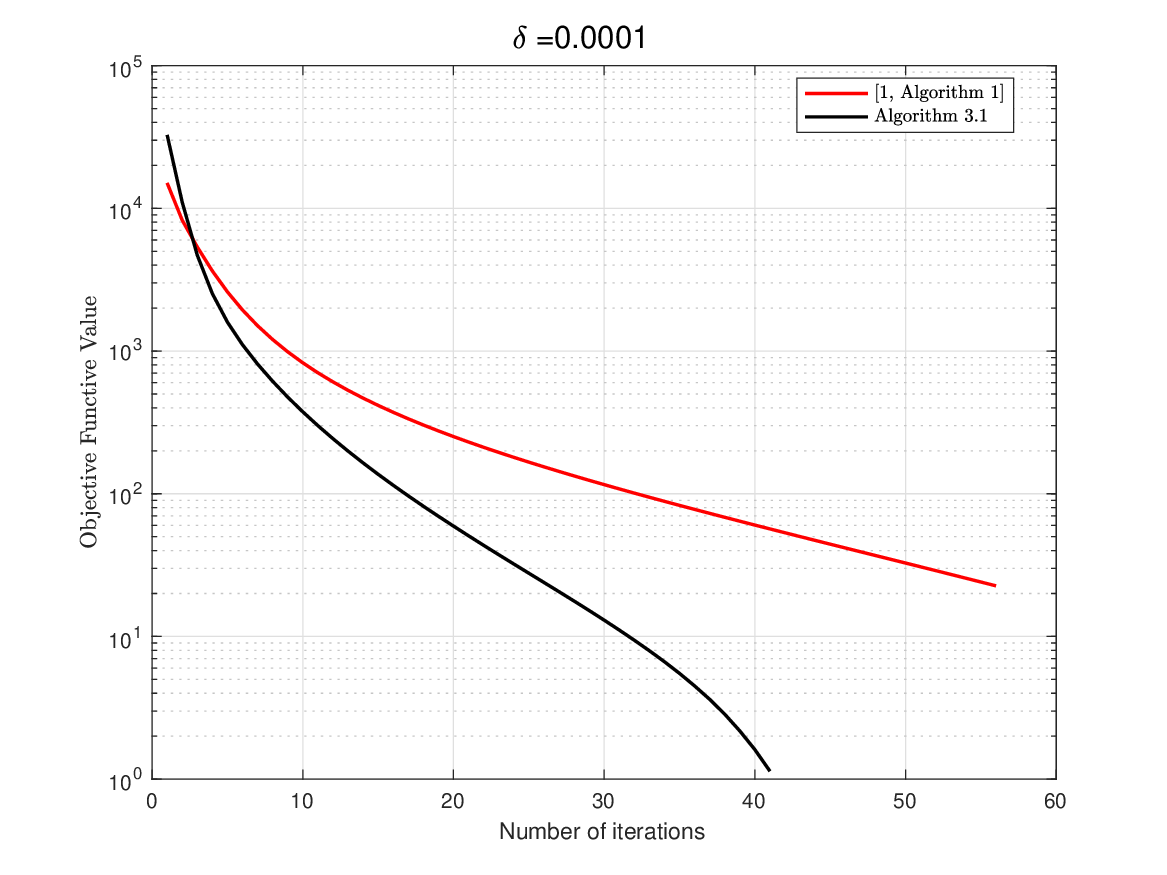}}}\hspace{5pt}
    \label{Fig5.2.4}
    \subfloat[]{
        \resizebox*{5.5cm}{!}{\includegraphics{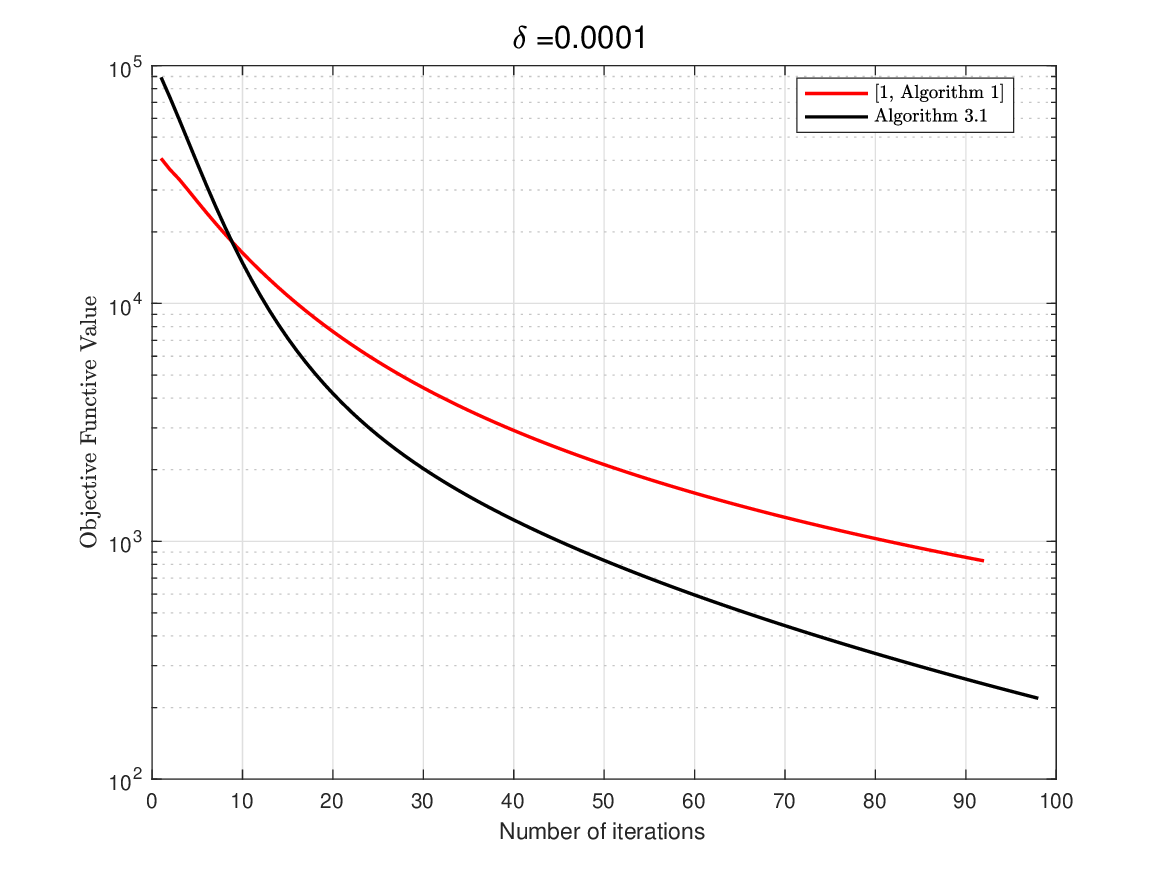}}}
    \label{Fig5.2.5}
     \subfloat[]{%
        \resizebox*{5.5cm}{!}{\includegraphics{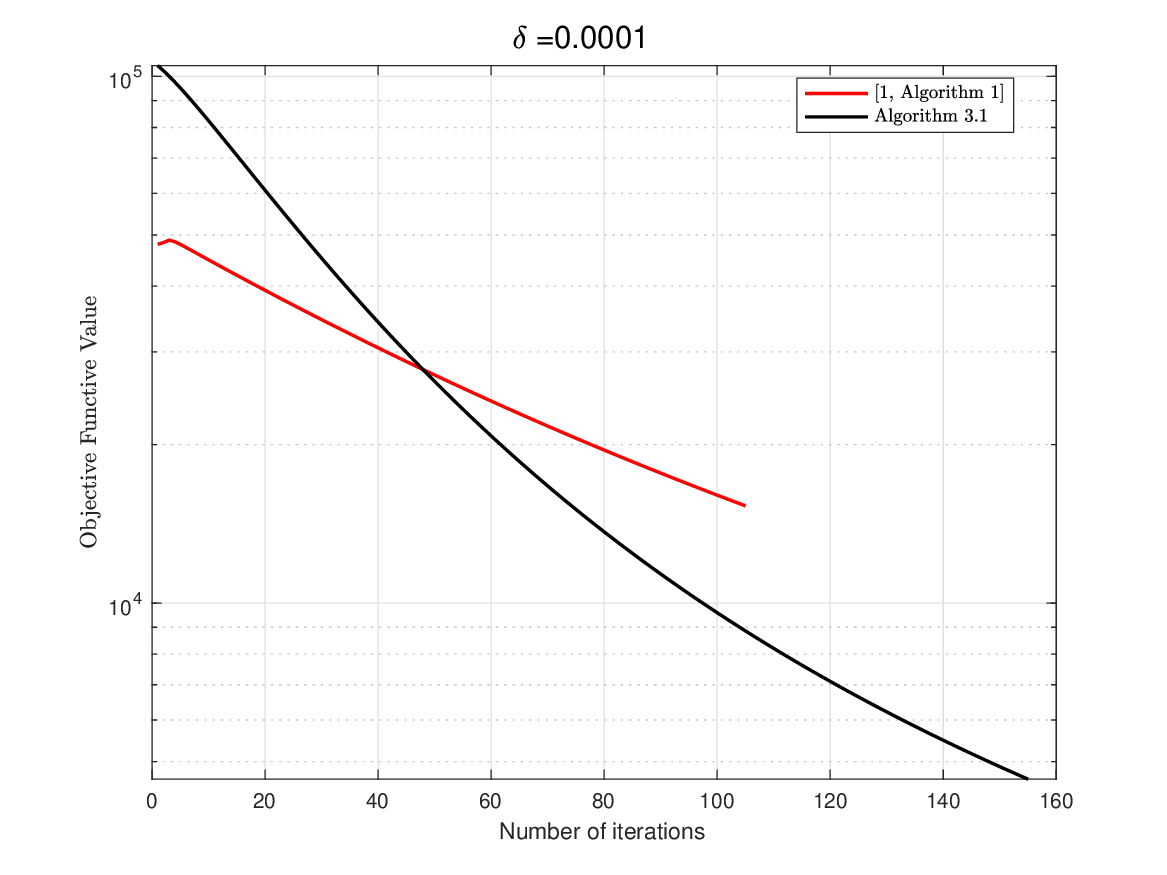}}}\hspace{5pt}
    \label{Fig5.2.6}

 \caption{The convergence behavior of the proposed Algorithm \ref{P4Al4.11} and of \cite[Algorithm 1]{YJS} for $\beta \in \{10^{-3},10^{-4},10^{-5}\}$ and $\delta =0.0001$.} \label{Fig5.2.22}
\end{figure}

\section*{Declarations}

\begin{itemize}
\item Conflict of interest/Competing interests: The authors declare that they have no competing interests.
\item Ethics approval and consent to participate: Not applicable.
\item Consent for publication: Not applicable.
\item Materials availability: Not applicable.
\item Code availability: The code used for the experiments is available from the corresponding author upon reasonable request.
\item Author contribution: All authors read and approved the final manuscript.
\end{itemize}


\bibliography{ref}

\end{document}